\documentclass[12pt,NoRunningHeads,reqno]{amsart}
\usepackage{latexsym}
\usepackage{amsmath}
\usepackage{amssymb}
\usepackage{amsfonts}
\usepackage{verbatim}
\usepackage[latin1]{inputenc}
\usepackage{color}

\newtheorem{theorem}{Theorem}
\newtheorem{lemma}{Lemma}
\newtheorem{definition}{Definition}
\newtheorem{corollary}{Corollary}
\newtheorem{Th}{Theorem}[section]

\newtheorem{rk}[Th]{Remark}

\let \ssection=\section
\renewcommand{\section}{\setcounter{equation}{0}\ssection}
\def\^#1{\if#1i{\accent"5E\i}\else{\accent"5E #1}\fi}
\def\"#1{\if#1i{\accent"7F\i}\else{\accent"7F #1}\fi}

\newcommand{\EE}{\mathbb E}

\newcommand{\PP}{\mathbb P}

\newcommand{\RR}{\mathbb R}
\newcommand{\CC}{\mathbb C}

\newcommand{\FF}{\mathcal F}
\newcommand{\ee}{\varepsilon}

\newcommand{\<}{\langle}
\renewcommand{\>}{\rangle}
\newcommand{\s}{\sigma}
\allowdisplaybreaks

\begin{document}
\title[On the splitting method]
{On the splitting method for Schr\"odinger-like evolution equations}
\author[Z. Brzezniak]
{Zdzislaw Brzezniak}
\thanks{ }
\address{Department of Mathematics, University of York, Heslington, York YO10 5DD, UK }
\email{zb500@york.ac.uk}

 \author[A. Millet]{Annie Millet}
\address
{Laboratoire de Probabilit\'es et Mod\`eles Al\'eatoires
(CNRS UMR 7599), Universit\'es Paris~6-Paris~7, Bo\^{\i}te Courrier 188,
4 place Jussieu, 75252 Paris Cedex 05,
{\it  and }\\
\indent Centre d'Economie de la Sorbonne (CNRS UMR 8174), \'Equipe
 SAMOS-MATISSE,
 Universit\'e Paris 1 Panth\'eon Sorbonne,
 90 Rue de Tolbiac, 75634 Paris Cedex 13}
\email{annie.millet@upmc.fr {\it and} annie.millet@univ-paris1.fr}

\begin{abstract}Using the approach of the splitting method developed by I.~Gy\"ongy and N.~Krylov
for parabolic quasi linear equations, we study the speed of convergence for general 
complex-valued stochastic
evolution equations.
 The approximation is given in general Sobolev spaces and the model considered here
contains both the parabolic quasi-linear equations under some (non strict)
 stochastic parabolicity condition as
well as linear Schr\"odinger equations
\end{abstract}

\subjclass[2000]{Primary  60H15, 65M12; Secondary 65M60, 65M15}

\keywords{Stochastic evolution equations, Schr\"odinger equation, splitting method,
speed of convergence, discretization scheme}

\maketitle

\section{Introduction} \label{intro}

Once the well-posedeness of a stochastic differential equation is proved, an important issue is to provide an efficient
way to approximate the unique solution. The aim of this paper is to propose  a fast converging scheme which gives
 a simulation of the trajectories of the solution  on a discrete time grid,  and in terms of some spatial approximation.
  The first results in this direction were obtained for Stochastic Differential Equations and it
 is well-known that the limit is sensitive to the approximation. For example, the Stratonovich integral
 is the limit of Riemann sums with the mid-pint approximation and the Wong Zakai approximation
 also leads to Stratonovich stochastic integrals, and not to the  It\^{o} ones in this finite dimensional framework.
There is a huge literature on this topic for stochastic PDEs,
mainly extending classical deterministic PDE methods to the stochastic framework.
Most of the papers deal with parabolic PDEs and take advantage of the smoothing effect of the
second order operator; see e.g. \cite{Gy1},  \cite{Prin}, \cite{Haus}, \cite{GM}, \cite{GM2},
\cite{GM2}, \cite{MM} and the references therein. The methods used in these papers are explicit, implicit or Crank-Nicholson
time approximations and the space discretization is made in terms of finite differences, finite elements
or wavelets. The corresponding speeds of convergence are the ''strong'' ones, that is uniform in time on some
bounded interval $[0,T]$ and with various functional norms for the space variable.
Some papers also study numerical schemes in other ``hyperbolic``  situations, such as the
KDV or Schr\"odinger equations as in \cite{DePr1}, \cite{dBD1} and \cite{dBD2}.  Let us also mention
the weak speed of convergence, that is of an  approximation of the expected value of a functional
of the solution by the similar one for the scheme obtained by \cite{dBD2} and \cite{DePr}. These references extend
to the infinite-dimensional setting a very crucial problem for finite-dimensional diffusion processes.

Another popular approach in the deterministic setting, based on  semi-groups theory, is the splitting method
which solves successively several evolution equations.  This technique has been used in a stochastic
case in a  series of papers by I. Gy\"ongy and N. Krylov. Let us especially mention reference \cite{GK}
which uses tools  from \cite{Pa}, \cite{KR-81} and \cite{KR-82}, and provides a very
elegant approach to study quasilinear evolution equations under (non strict)
stochastic parabolicity conditions. In their framework, the smoothing effect of the second operator is exactly
balanced by the quadratic variation of the stochastic integrals, which implies that there is no increase
of space regularity with respect to that of the initial condition. Depending on the number of
steps of the splitting, the speed of convergence is at least twice that of the classical finite differences
or finite elements methods.  A series of papers has been using
the splitting  technique in the linear and non-linear cases for the deterministic Schr\"odinger equation; see e.g. \cite{BM},
\cite{DF}, \cite{NT} and the references therein.

The stochastic Schr\"odinger equation studies  complex-valued processes where the second order
 operator $i \Delta$ does not improve (nor decay) the space regularity of the solution with respect
 to that of the initial condition. Well-posedeness of this equation has been proven in a non-linear
 setting by A. de Bouard and A. Debussche \cite{dBD2}; these authors have also studied
 finite elements discretization schemes for the corresponding solution under conditions stronger than that in \cite{dBD1}.

 The aim of this paper is to transpose the approach from \cite{GK} to general  quasilinear complex-valued
 equations including both the ``classical degenerate``  parabolic setting as well as
 the quasilinear Sch\"odinger equation. Indeed, the method used in \cite{GK} consists
 in replacing the usual splitting via  semi-groups arguments  by the study of $p $-th  moments of $Z^{0}- Z^{1}$
 where $Z^{0}$ and $Z^{1}$ are solutions of two stochastic
 evolution equations with the same driving noise and different families of  increasing processes
 $V_{0}^{r}$ and $V_{1}^{r}$ for $r=0, 1, \cdots, d_{1}$ driving the drift term.
 It does not extend easily to nonlinear drift terms because it is based on some   linear interpolation
  between the two cases $V_{0}^{r}$ and $V_{1}^{r}$.
 Instead of getting an upper estimate of the $p$-th moment of $Z^{0}- Z^{1}$
 in terms of the total variation of the measures defined by the differences
 $V_{0}^{r}- V_{1}^{r}$, using  integration by parts they obtain an upper estimate in terms of the sup norm of the process
 $(V_{0}^{r}(t)- V_{1}^{r}(t), t\in [0,T])$.

 We extend this model as follows:
  given  second order linear differential operators $L^{r}$, $r=0, \cdots, d_{1}$ with complex coefficients,
  a finite number of  sequences of first order linear operators $S^{l}$, $l\geq 1$
with complex coefficients, a   sequence
 of real-valued martingales $M^{l}$, $l\geq 1$ and a finite number of  families of real-valued
increasing processes $V_{i}^{r}$, $i\in \{0,1\}$, $r=0, \cdots,  d_{1}$,
  we consider the following system of stochastic evolution
equations 
\[ dZ_{i}(t) = \sum_{r}  L_{r}(t,\cdot)  Z_{i}(t)  dV^{r}_{i}(t) + 
 \sum_{l}  S_{l}(t,\cdot) Z_{i}(t) dM^{l}(t), \;\; i=1, \cdots, d,
\]
with an initial condition $Z_{i}(0)$ belonging to  the  Sobolev space $H^{m,2}$  for a certain  $m\geq 0$.
Then under proper assumptions on the various coefficients and processes, under which
a stochastic parabolicity condition (see Assumptions {\bf(A1)}- {\bf(A4(m,p))}
in section \ref{apriori}), we prove that for $p\in [2,\infty)$, we have
\begin{equation} \label{speedIntro}
 \EE \Big( \sup_{t\in [0,T]} \|Z_{1}(t) - Z_{0}(t)\|_{m}^{p}\Big) \leq C \Big( \EE\|Z_{1}(0)-Z_{0}(0)\|_m^{p} +
A^{p}\Big) ,
\end{equation}
where  $A = \sup_{\omega} \, \sup_{t\in [0,T]}\, \max_{r}  |V^{r}_{1}(t) - V^{r}_{0}(t)|$.
When the operator $L_{r}= i\Delta + \tilde{L}^{r}$ for certain first order differential operator $\tilde{L}_{r}$, we obtain
the quasilinear Schr\"odinger equation. Note that in this case, the diffusion operators $S^{l}$ are linear
 and cannot contain first order derivatives.

As in \cite{GK}, this abstract result yields the speed of convergence of the following splitting method.
Let $\tau_{n}= \{ iT/n, i=0, \cdots, n\}$ denote a time grid on $[0,T]$ with constant mesh $\delta = T/n$
and define the increasing processes $A_{t}(n)$ and $B_{t}(n)= A_{t+\delta}(n)$, where
\[ A_{t}(n)= \left\{ \begin{array}{ll}
k\delta & \; \mbox{\rm for }  t\in [2k\delta , (2k+1)\delta] , \\
t-(k+1)\delta& \; \mbox{\rm for }  t\in [(2k+1)\delta, (2k+2)\delta] .
\end{array}
\right.
\]
  Given a time-independent second order  differential operator $L$, first order time-independent
operators $S^{l}$ and   a sequence $(W^l, l\geq 1)$ of independent
one-dimensional Brownian motions, let
$Z$, $Z_{n}$ and $\zeta_{n}$ be solutions to the evolution equations
\begin{align*}
dZ(t) &=  L Z(t) dt +\sum_{l} S_{l} Z(t) \circ dW^{l}_{t},\\
dZ_{n}(t) & =  L Z_{n}(t) dA_{t}(n) +\sum_{l} S_{l} Z_{n}(t) \circ dW^{l}_{B_{t}(n)},\\
d\zeta_{n}(t) &=  L \zeta_{n}(t) dB_{t}(n) +\sum_{l} S_{l}(t,\cdot) \zeta_{n}(t)\circ  dW^{l}_{B_{t}(n)},
\end{align*}
where $\circ dW^l_t$ denotes the Stratonovich integral.
 The Stratonovich integral is known to be the right one to ensure stochastic parabolicity
when the differential operators $S_{l}$ contain first order partial derivatives (see e.g. \cite{GK}).
Then $\zeta_{n}(2k\delta,x)= Z(k\delta,x)$, while the values of $Z_{n}(2k\delta,x)$ are those
of the process $\tilde{Z}_{n}$ obtained by the following spitting method: one solves successively the correction
and prediction  equations on each time interval
$[iT/n, (i+1)T/n)]$:
$dv_{t}=  L v_{t} dt$  and then  $d\tilde{v}_{t} = \sum_{l} S^{l}(t,\cdot) \tilde{v}_{t}\circ dW^{l}_{t}$.
 Then, one has $A=CT/n$, and we deduce that
$\EE\Big( \sup_{t\in [0;T]}  \|Z(t) - \tilde{Z}_{n}(t)\|_{m}^{p}\Big) \leq C n^{-p}$.
As in \cite{GK}, a $k$-step splitting would yield a rate of convergence of $C n^{-kp}$.

The paper is organized as follows. Section \ref{apriori} states the model, describes the evolution
equation, proves well-posedeness as well as apriori estimates.
  In section \ref{speedabstract}  we prove  \eqref{speedIntro} first in the case of
 time-independent coefficients
of the differential operators, then in the general case under more regularity  conditions.
As explained above, in section \ref{sspeed} we deduce
 the rate of convergence of the splitting method for evolution equations
generalizing the quasi-linear Schr\"odinger equation.
The speed of convergence of the non-linear Schr\"odinger equation will be addressed in a forthcoming paper.
As usual, unless specified otherwise,  we will denote by $C$ a constant which  may change from one line to the next.
\section{Well-posedeness and first apriori estimates}\label{apriori}
\subsection{Well-posedeness results}
Fix $T>0$, $\mathbb{F}=(\FF_t, t\in [0,T])$ be a filtration on the probability space $( \Omega,\FF,\PP)$
 and consider the following ${\CC}$-valued evolution equation on the process $Z(t,x)= X(t,x) +i Y(t,x)$
 defined for $t\in [0,T]$ and $x\in \RR^d$:
\begin{align} \label{Z}
dZ(t,x)&=\sum_{r=0}^{d_1} \big[ L_r Z(t,x) + F_r(t,x)]\, dV^r_t
 + \sum_{l\geq 1} \big[ S_l Z(t,x) + G_l(t,x)\big] dM^l_t,\\
Z(0,x)&=Z_0(x)=X_0(x)+iY_0(x), \label{CIZ}
\end{align}
where $d_1$ is a positive integer,  $ (V^r_t, t\in [0,T])$, $r=0, 1, \cdots, d_1$  are
 real-valued increasing processes,
 $\big( M^l_t, t\in [0,T]\big)$, $l\geq 1$,  are independent real-valued $(\FF_t, t\in [0,T])$-martingales,
 $L_r$ (resp.  $S_l$) are
second (resp.  first) order differential operators defined as follows:
\begin{eqnarray}
L_r Z(t,x)&=&\sum_{j,k=1}^{d} D_k\Big( \big[a^{j,k}_r(t,x) + i b^{j,k}_r(t)\big] D_jZ(t,x) \Big)
+ \sum_{j=1}^{d} a^j_r(t,x) D_j Z(t,x),  \nonumber \\
&& + \big[ a_r(t,x)+ib_r(t,x)\big] Z(t,x) , \label{Lr} \\
S_l Z(t,x)&=& \sum_{j=1}^{d}  \sigma_l^j(t,x)  D_j Z(t,x)
+ \big[ \sigma_l(t,x)+i\tau_l(t,x)\big] Z(t,x).
\label{Sl}
\end{eqnarray}
Let $m\geq 0$ be an integer. Given $\CC$-valued functions $Z(.)=X(.)+iY(.)$
 and $\zeta(.)= \xi(.)+i\eta(.) $ which belong to  $L^2(\RR^d)$, let
\[ (Z,\zeta) = (Z, \zeta)_0:= \int_{\RR^d} Re\big( Z(x) \overline{\zeta(x)}\big) dx = \int_{\RR^d}
 \big[ X(x)\xi(x)+Y(x)\eta(x) \big] dx.\]
Thus,  we have
$(X,\xi)=\int_{\RR^d} X(x) \, \xi(x)\, dx$, so that $(Z,\zeta) = (X,\xi) + (Y,\eta)$.
For any multi-index $\alpha = (\alpha_1, \alpha_2, \cdots, \alpha_d)$ with non-negative integer components
 $\alpha_i$, set $|\alpha|= \sum_j \alpha_j$
and for a regular enough function $f:\RR^d\to \RR$, let $D^\alpha f$ denote the partial derivative
$(\frac{\partial}{\partial x_1})^{\alpha_1}  \cdots
(\frac{\partial}{\partial x_d})^{\alpha_d} (f)$.
For $k=1, \cdots, d$, let  $D_kf$ denote the partial
derivative $\frac{\partial f}{\partial x_k}$.  For a $\CC$-valued function $F=F_1+i F_2$ defined on $\RR^d$, let
$D^\alpha F = D^\alpha F_1 +i D^\alpha F_2$ and $D_kF=D_kF_1+iD_k F_2$.
 Finally, given a positive integer $m$,
 say that $F\in H^m$ if and only if $F_1$ and $F_2$
belong to the (usual) real Sobolev space $H^m=H^{m,2}$. Finally, given $Z=X+iY$ and $\zeta=\xi+i\eta$ which belong to $H^m$,
set
\begin{align} \label{prodm}
(Z,\zeta)_m &
 =     \sum_{\alpha : 0\leq |\alpha| \leq m}  \int_{\RR^d}
Re\big( D^\alpha Z(x) \overline{D^\alpha\zeta(x)} \big) dx  \\
& =  \sum_{\alpha : 0\leq |\alpha| \leq m} \int_{\RR^d}
\big[ D^\alpha X(x) D^\alpha \xi(x)  + D^\alpha Y(x) D^\alpha \eta(x)\big] dx, \nonumber   \\
  \| Z\|_m^2&= (Z,Z)_m= \sum_{\alpha : |\alpha|\leq m} \int_{\RR^d}
Re \big( D^\alpha Z(x)\overline{D^\alpha Z(x)}\big) dx. \label{normHm}
\end{align}
 We suppose that the following assumptions are satisfied:
\smallskip

\noindent {\bf Assumption (A1)}  For  $r=0, \cdots, d_1$,  $(V^r_t, t\in [0,T])$  are predictable increasing
processes. There exist a positive constant $\tilde{K}$ 
 and an increasing predictable process
 $(V_t, t\in [0,T])$ such that:
\begin{align}
&  V_0=V_0^r=0,\;  r=0, \cdots d_1,\quad
 V_T\leq \tilde{K} \; \mbox{\rm  a.s.,} \label{V1}\\
&  \sum_{r=0}^{d_1} \! dV^r_t + \sum_{l\geq 1} d\langle M^l\rangle_t \leq dV_t   \quad
\mbox{\rm   a.s. in the sense of measures}. \label{V}
\end{align}

\noindent {\bf Assumption (A2)}\\
 \indent (i) For $r=0, 1, \cdots d_1$, the matrices $(a^{j,k}_r(t,x), j,k=1, \cdots d)$ and  $(b^{j,k}_r(t),
 j,k=1, \cdots d)$   are $(\FF_t)$-predictable 
 real-valued symmetric for almost every $\omega$, $t\in [0,T]$
 and $x\in \RR^d$.

(ii) For every $t\in (0,T]$ and $x,y \in \RR^d$
\begin{equation} \label{coercive}
\sum_{j,k=1}^d \!\! y^j y^k \Big[ 2\sum_{r=0}^{d_1}  a^{j,k}_r(t,x) \, dV^r_t - \sum_{l\geq 1}
\sigma_l^j(t,x) \sigma_l^k(t,x)   d\langle M^l\rangle_t\Big]  \geq 0
\end{equation}
 a.s. in the sense of measures.
\smallskip

\noindent {\bf Assumption (A3(m))}
There exists a constant $\tilde{K}(m)$ such that for all $j, k=1, \cdots d$, $r=0, \cdots, d_1$, $l\geq 1$,
any multi-indices $\alpha$ (resp. $\beta$) of length $|\alpha|\leq m+1$ (resp. $|\beta|\leq m$), and 
for every $(t,x)\in (0,T]
\times \RR^d$ one has a.s.
\begin{align}
|D^\alpha a^{j,k}_r(t,x)|+ |b^{j,k}(t)|+ |D^\alpha a^{j}_r(t,x)| + |D^\beta a_r(t,x)|+ |D^\beta b_r(t,x)|
\leq \tilde{K}(m) ,  \label{majoab} \\
|D^\alpha \sigma_l^j(t,x)|+  |D^\alpha \sigma_j(t,x)|+ |D^\alpha \tau_l(t,x)| \leq \tilde{K}(m).
 \label{majost}
\end{align}

\noindent {\bf Assumption (A4(m,p))} Let  $p\in [2,+\infty)$;
for any $r=0, \cdots, d_1$, $l\geq 1$, the processes
$F_r(t,x)=F_{r,1}(t,x)+i F_{r,2}(t,x)$ and $G_l(t,x)=G_{l,1}(t,x)+i G_{l,2}(t,x)$
are predictable, $F_{r}(t,\cdot)\in H^m$ and $G_{r}(t,\cdot)\in H^{m+1}$. Furthermore, if we denote
\begin{equation}\label{defK}
K_m(t) = \int_0^t \Big[ \sum_{ r=0}^{d_1} \|F_{r}(s)\|_m^2  \, dV^r_s
+ \sum_{l\geq 1}  \| G_l(s)\|_{m+1}^2  d\langle M^l\rangle_s \Big],
\end{equation}
then 
\begin{equation}\label{bndpini}
\EE \big( \|Z_0\|_m^p + K_m^{\frac{p}{2}}(T)\big) <+\infty.
\end{equation}

The following defines what is considered to be a (probabilistically strong) weak solution of the evolution equations
\eqref{Z}-\eqref{CIZ}.
\begin{definition}\label{defsol}
A $\CC$-valued $(\FF_t)$-predictable process $Z$ is a solution to the evolution equation \eqref{Z} with initial condition
$Z_0$ if  
\[ \PP\Big(\int_0^T \|Z(s)\|_1^2 ds <+\infty\Big)=1, \quad \EE \int_0^T |Z(s)|^2 dV_s <\infty, \]
 and for every $t\in [0,T]$ and $\Phi=\phi+i\psi$, where $\phi$ and
$\psi$ are  ${\mathcal C}^\infty$ functions with compact support from $\RR^d$ to $\RR$, one has a.s.
\begin{align}\label{sol}
(Z(t&),\Phi) = (Z(0),\Phi) + \sum_{r=0}^{d_1}\int_0^t
 \Big[ -\sum_{j,k=1}^{d} \Big( \big[a^{j,k}_r(s,\cdot)+ib^{j,k}_r(s)\big] D_jZ(s,\cdot) ,  D_k\Phi \Big) \nonumber  \\
& + \sum_{j=1}^d  \Big(a^j_r(s,\cdot)  D_j Z(s,\cdot) + [a_r(s,\cdot)+ib_r(s,\cdot)] Z(s,\cdot) + F_r(s,\cdot)
 ,  \Phi\Big)\Big]
dV_r(s) \nonumber \\
&+ \sum_{l\geq 1} \int_0^t \big( S_l(Z(s,\cdot)) + G_l(s,\cdot) , \Phi)\, dM^l_s.
\end{align}
\end{definition}

\begin{theorem} \label{Hmestimates}
Let $m\geq 1$ be an integer and suppose that Assumptions {\bf (A1),(A2),  (A3(m))} and {\bf (A4(m,2))} (i.e., for $p=2$)
 are satisfied.

(i) Then equations \eqref{Z} and \eqref{CIZ} have a unique solution $Z$, such that
\begin{equation}\label{apriori1}
\EE\Big(\sup_{t\in [0,T]} \|Z(t,\cdot)\|_m^2 \Big) \leq C\, \EE\Big(\|Z_0\|^2_m + K_m(T)\Big)<\infty ,
\end{equation}
for a  constant $C$ that only depends on the constants which appear in the above listed conditions.
Almost surely, $Z\in{\mathcal C}([0,T], H^{m-1})$ and almost surely the map $[0,T]\ni t\mapsto Z(t,\cdot)\in H^m$ is weakly continuous.

(ii) Suppose furthermore that Assumption  {\bf (A4(m,p))} holds for $p\in (2,+\infty)$. Then
there exists a constant $C_p$ as above such that
 \begin{equation}\label{apriori2}
\EE\Big(\sup_{t\in [0,T]}\|Z(t,\cdot)\|_m^p\Big) \leq C_p \, \EE\Big(\|Z_0\|^p_m + K_m^{p/2}(T)\Big).
\end{equation}
\end{theorem}
\begin{proof} 
Set ${\mathcal H}=H^m$, ${\mathcal V}=H^{m+1}$ and ${\mathcal V}'=H^{m-1}$. Then ${\mathcal V}
\subset {\mathcal H}\subset {\mathcal V}'$ is a Gelfand triple  for the equivalent norm
$ \vert(I-\Delta)^{m/2}u \vert_{L^2}$ on the space $H_{m}$.
 Given $Z=X+iY \in {\mathcal V}$
and $\zeta=\xi+i\eta\in {\mathcal V}'$ set
\[ \< Z,\zeta\>_m = \<X,\xi\>_m + \< Y,\eta\>_m, \quad \mbox{\rm and } \; \< Z,\zeta\>:=\< Z,\zeta\>_0, \]
where $\<X,\xi\>_m$ and $\< Y,\eta\>_m$ denote the duality  between  the (real) spaces $H^{m+1}$ and $H^{m-1}$.
For every multi-index  $\alpha$,   let
\[  
 {\mathcal I}(\alpha)= \{(\beta,\gamma)\; :\;  \alpha= \beta+\gamma,\;  |\beta|, |\gamma| \in
\{0, \cdots, |\alpha|\}\, \} .
\]  
 To ease notations, we skip the time parameter when writing the coefficients $a_r$, $b_r$, $\sigma_l$ and
$\tau_l$.  Then for $l\geq 1$, using the Assumption {\bf (A3(m))}, we deduce that if $Z=X+iY\in H^{m+1}$, we have
for $|\alpha|\leq m$,
\begin{equation} \label{DalphaS}
 D^\alpha [S_l(Z)]=  \sum_{ j=1}^d \sigma_l^j(x) \,  D_jD^\alpha Z
 +  \sum_{(\beta,\gamma)\in {\mathcal I}(\alpha)}  C_l(\beta, \gamma)  D^\beta Z ,
\end{equation}
 with  functions  $C_l(\beta,\gamma)$  from $\RR^d$ to $\CC$ such that
 $\sup_{l\geq 1} \sup_{x\in \RR^d} |C_l(\beta,\gamma)(x)|
<+\infty$.
A similar computation proves that for every multi-index $\alpha$ with $|\alpha|\leq m$, $r=0, \cdots, d_1$
\begin{equation} \label{DalphaL}
 D^\alpha [L_r(Z)]= L_r ( D^\alpha Z )
 + \sum_{j,k=1}^{d} \sum_{(\beta,\gamma)\in {\mathcal I}(\alpha) : |\gamma|=1} \!\!\!\!\!\!\!\!\! D_k\Big( D^\gamma a^{j,k}_r
D_j D^\beta Z\Big) +  \sum_{|\beta|\leq m}\! \! C_r( \beta, \gamma) D^\beta Z,
\end{equation}
for some bounded functions $C_r(\beta,\gamma)$ from $\RR^d$ into $\CC$.
Hence for every $r=0, \cdots, d_1$, one has a.s.  $L_r:{\mathcal V}\times \Omega \to {\mathcal V}'$
and similarly, for every $l\geq 1$, a.s.
 $S_l:{\mathcal V} \times \Omega \to {\mathcal H}$.

For every $\lambda>0$ and    $Z=X+iY \in {\mathcal H}$, let us set
\begin{equation} \label{Lrl}
 L_{r,\lambda}Z:= L_rZ+ \lambda(\Delta X +i \Delta Y) = L_r Z  + \lambda \Delta Z .
\end{equation}
Consider the evolution equation for the process  $Z^\lambda(t,x)=X^\lambda(t,x) +iY^\lambda(t,x)$,
\begin{align} \label{Ze}
dZ^\lambda(t,x)&=\sum_{r=0}^{d_1}  \big[ L_{r,\lambda}  Z^\lambda(t,\cdot)
 +  F_r(t,x)]\, dV^r_t  
  + \sum_{l\geq 1}  \big[ S_l Z^\lambda(t,x) + G_l(t,x)\big] dM^l_t,\\
Z^\lambda(0,x)&=Z_0(x)=X_0(x)+iY_0(x) . \label{CIZe}
\end{align}
In order to prove well-posedeness of the problem \eqref{Ze}-\eqref{CIZe},
 firstly we have to  check the following stochastic parabolicity condition:

\noindent {\bf Condition (C1)} There exists a constant $K>0$ such that for $Z\in H^{m+1}$, $t\in [0,T]$: 
\[
2\sum_{r=0}^{d_1}  \< L_rZ \, ,\, Z \>_m \,  dV^r_{t}  + \sum_{l\geq 0}
\|S_l(Z)\|_m^2 \, d\<M^l\>_t \leq K \|Z\|_m^2  \,  dV_t
\]
a.s. in the sense of measures.

Let $Z=X+iY \in H^{m+1}$;
using  \eqref{DalphaL} and
\eqref{DalphaS}, we deduce that
\begin{equation}  2\sum_{|\alpha|=m} \sum_{r=0}^{d_1} \< D^\alpha L^r Z\, ,\, D^\alpha Z \> \, dV^r_t
 + \sum_{|\alpha|=m} \sum_{l\geq 1}
|D^\alpha S_l Z|^2 \, d\<M^l\>_t = \sum_{\kappa=1}^{5} \, d T_\kappa(t),\label{decomp1}
\end{equation}
where to ease notation we  drop the time index in the right handsides and we set:
\begin{align*}
  dT_1(t)=  & \sum_{|\alpha|=m} \sum_{j,k=1}^d \!\!  \Big\{  -2 \sum_{ r=1}^{d_1}\!\!
\Big[ \big( a^{j,k}_r D_jD^\alpha X , D_kD^\alpha X\big)  
-  \big( b^{j,k}_r D_jD^\alpha Y , D_kD^\alpha X\big)  \\
& 
+ \big( a^{j,k}_r D_jD^\alpha Y , D_kD^\alpha Y\big) 
+  \big( b^{j,k}_r D_jD^\alpha X , D_kD^\alpha Y\big) \Big] dV^r_t \\
&
+ \sum_{l\geq 1}  \Big[
\big( \sigma^j_l D_jD^\alpha X , \sigma_l^k D_k D^\alpha X \big)
+ \big( \sigma^j_l D_jD^\alpha Y , \sigma_l^k D_k D^\alpha Y\big)
\Big] d\< M^l \>_t \Big\} , 
\\ 
dT_2(t)= &-2\sum_{r=0}^{d_1} \; \sum_{|\alpha|\leq m} \;
  \Big\{ \sum_{j,k=1}^{d} \; \sum_{(\beta,\gamma)\in {\mathcal I}(\alpha): |\gamma|=1}  \\
& \quad\quad  \Big[ \big(D^\gamma a^{j,k}_r D_jD^\beta X , D_kD^\alpha X\big)  
 + \big(D^\gamma a^{j,k}_r D_jD^\beta Y , D_kD^\alpha Y\big)\Big] \nonumber \\
& 
  + \sum_{j=1} ^{d } \Big[
\big( a_r^j D_jD^\alpha X , D^\alpha X\big) + \big( a_r^j D_jD^\alpha Y , D^\alpha Y\big)\Big] \Big\} dV^r_t  ,
 \\ 
 dT_3(t)=& \sum_{l\geq 1} \; \sum_{j,k=1}^{d} \; \sum_{|\alpha|=m}\;
\sum_{(\beta,\gamma)\in {\mathcal I}(\alpha) : |\gamma|=1} \Big[
\big( D^\gamma \sigma_l^j D_jD^\beta X , \sigma_l^k D_k D^\alpha X\big)  \\
 & 
  + ( D^\gamma \sigma_l^jD_jD^\beta Y , \sigma_l^k D_k D^\alpha Y\big)
 d\< M^l\>_t  ,  \\ 
dT_4(t)=& \sum_{l\geq 1} \; \sum_{j,k=1}^{d} \; \sum_{|\alpha|=m}\;  \Big[
(\sigma_l^j D_j D^\alpha X , \sigma_l D^\alpha X)
- (\sigma_l^j D_j D^\alpha X , \tau_l D^\alpha Y)  \\
&
 + (\sigma_l^j D_j D^\alpha Y , \tau_l D^\alpha X)
 + (\sigma_l^j D_j D^\alpha Y , \sigma_l D^\alpha Y)
\Big]  d\< M^l\>_t  ,  \\ 
 dT_5(t)= & \sum_{|\alpha|\vee |\beta| \leq m} \; \Big\{  \sum_{r=0}^{d_1} \; \Big[ \sum_{j,k=1}^{d}
\big( C_r^{j,k}(.) D^\beta Z , D^\alpha Z\big) \nonumber \\
& + \sum_{j=1}^{d}
\big( C_r^{j}(.) D^\beta Z , D^\alpha Z\big)
 + \big( C_r(.)   D^\beta Z , D^\alpha Z\big)\Big] dV^r_t \nonumber \\
& + \sum_{l\geq 1} \! \Big[ \!   \sum_{j=1}^{d}
\Big( \Big[  \tilde{C}_l(.) +  \sum_{l\geq 1} \tilde{C}^j_l(.)\Big] ]   D^\beta Z , D^\alpha Z\Big) \Big] d\< M^l\>_t\Big\},
\end{align*}
where $C_r^{j,k}$, $C_r^{j}$, $C_r$, $\tilde{C}_l^{j}$ and $\tilde{C}_l$ are bounded functions from $\RR^d$ to
$\CC$ due to Assumption {\bf (A4(m,p))} for any $p\in [2,\infty)$.

For every $r$  the matrix $b_r$ is symmetric; hence
 $\sum_{j,k} \big[ (b^{j,k}_r D_j D^\alpha X , D_k D^\alpha Y)
- (b_r^{j,k} D_jD^\alpha Y , D_k D^\alpha X)\big] =0$.
 Hence,
Assumption {\bf (A2)} used with $y_j=D_jD^\alpha X$ and  with $y_j=D_jD^\alpha Y$, $j=1, \cdots, d$,  implies $T_1(t)\leq 0$
for $t\in [0,T]$. Furthermore, Assumption {\bf (A3(m))}  yields the existence of a constant $C>0$ such that
 $dT_5(t) \leq C \|Z(t)\|_m^2 dV_t$ for all $t\in [0,T]$.
Integration by parts shows that for regular enough functions $f,g,h :\RR^d\to \RR$, $(\beta,\gamma)\in
   {\mathcal I}(\alpha)$ with $|\gamma|=1$,  we have
\begin{equation} \label{*}
 ( fD^\beta g, D^\alpha h) = - ( f D^\alpha g, D^\beta h\> - \< D^\gamma f D^\beta g , D^\beta h\>.
\end{equation}
Therefore, the symmetry of the matrices $a_r$ implies that for $\phi\in \{X(t),Y(t)\}$
 and $r=0, \cdots, d_1$,
\[\sum_{j,k=1}^{d} \big(D^\gamma a^{j,k}_r D_j D^\beta \phi , D_k D^\alpha \phi  \big)  =
-  \frac{1}{2} \sum_{j,k=1}^{d}  \big( D^\gamma (D^\gamma  a^{j,k}_r)  D_j D^\beta \phi , D_k D^\beta \phi \big) . \]
A similar argument proves that for fixed $j=1, \cdots, d$, $r=0, \cdots, d_1$  and $\phi=X(t)$ or $\phi=Y(t)$,
\[( a_r^j D_j D^\alpha \phi , D^\alpha \phi) = - \frac{1}{2} ( D_j a^j_r D^\alpha \phi , D^\alpha \phi).\]
Therefore,  Assumption {\bf (A3(0))}  implies  the existence of $K>0$ such that
$  dT_2(t)  \leq  K  \|Z(t)\|_m^2 dV_t$
for all $t\in [0,T]$. Furthermore,  $dT_4(t)$ is the sum of terms $(\phi \psi , D_j \phi)$ $ d\<M^l\>_t$
where $\phi=D^\alpha X(t)$ or $\phi= D^\alpha Y(t)$, and $\psi=fg$, with
$f\in \{\sigma_l^k \}$ and $g\in \{ \sigma_l,\tau_l\}$. The identity
$(\phi \psi , D_j \phi) = -\frac{1}{2} (\phi D_j\psi , \phi)$,  which is easily deduced from  integration by parts,
and Assumptions {\bf (A1)} and  {\bf (A3(m))} imply the existence of $K>0$ such that
$dT_4(t) \leq K \|Z(t)\|_m^2 dV_t$ for every $t\in [0,T]$.
The term  $dT_3(t)$ is the sum over $l\geq 1$ and multi-indices $\alpha$
with $|\alpha |=m$ of
\[ A(l,\alpha) = \sum_{j,k=1}^{d} \;
 \sum_{(\beta,\gamma)\in {\mathcal I}(\alpha): |\gamma|=1}\big( D^\gamma f^j_l f^k_l D_j D^\beta \varphi ,
D_k D^\alpha \varphi\big) ,\]
with $\varphi = X(t)$ or $\varphi = Y(t)$ and $f_l^j=\sigma_l^j$  for every $j=1, \cdots, d$.
Then, $A(l,\alpha)=B(l,\alpha) - C(l,\alpha)$, where  $ B(l,\alpha)=\sum_{j,k=1}^{d} B_{j,k}(l,\alpha)$ and
\begin{align*}  B_{j,k}(l,\alpha) =&   \sum_{(\beta,\gamma)\in {\mathcal I}(\alpha): |\gamma|=1}
\big(  D^\gamma(f^j_l f^k_l) D_j D^\beta \varphi ,
D_k D^\alpha \varphi \big) , \\
C(l,\alpha)= & \sum_{j,k=1}^{d} \sum_{(\beta,\gamma)\in {\mathcal I}(\alpha): |\gamma|=1}
\big(  D^\gamma f^k_l f^j_l D_j D^\beta \varphi , D_k D^\alpha \varphi\big) .
\end{align*}
Integrating by parts twice and exchanging the partial derivatives $D_j$ and $D_k$
in each term of the sum in $C(l,\alpha)$,
we deduce that
\begin{align*} &
\big(  D^\gamma f^k_l f^j_l D_j D^\beta \varphi , D_k D^\alpha \varphi\big)  =
 - \big( \big[ D_k [ D^\gamma f^k_l f^j_l ]  D_j D^\beta \varphi 
+  D^\gamma f^k_l f^j_l D_k  D_j D^\beta \varphi\big]  ,  D^\alpha \varphi\big) \\& =
  \big(  -  D_k [ D^\gamma f^k_l f^j_l ]  D_j D^\beta \varphi 
+   \big( D_j [ D^\gamma f^k_l f^j_l ]  D_k D^\beta \varphi  ,  D^\alpha \varphi\big)\\
&\qquad\qquad   + \big(  D^\gamma f^k_l f^j_l   D_k D^\beta \varphi ,
D_j D^\alpha \varphi\big).
\end{align*}
On the other hand, by symmetry  we obviously have
\[ \sum_{j,k}  \big( D^\gamma f^j_l f^k_l D_k D^\beta \varphi ,
D_j D^\alpha \varphi\big)  =
\sum_{j,k} \big(  D^\gamma f^k_l f^j_l D_j D^\beta \varphi , D_k D^\alpha \varphi\big) .\]
Using  Assumptions {\bf (A1)} and  {\bf (A3(1))} we deduce that there exist bounded functions
 $\phi(\alpha,\tilde{\alpha},l)$ defined for multi-indices $\tilde{\alpha}$ which have at most
one component different from those of $\alpha$, and  such that
\[ A(l,\alpha) = \frac{1}{2} B(l,\alpha) +
 \sum_{|\tilde{\alpha}|=m} \big( \phi(\alpha, \tilde{\alpha},l) D^{\tilde{\alpha}} \Phi ,
D^\alpha \Phi\Big).\]
Furthermore, 
 integration by parts yields
\[ \sum_{j,k=1}^{d} B_{j,k}(l,\alpha)=-\frac{1}{2} \sum_{j,k=1}^{d}
 \sum_{(\beta,\gamma)\in {\mathcal I}(\alpha): |\gamma|=1}
\big( D^\gamma D^\gamma [f^j_l f^k_l] D_j D^\beta \varphi , D_k  D^\beta \varphi\big)  .\]
Thus, Assumption {\bf (A3(1))} implies the existence of  a constant $C>0$ such that
 for the various choices of $\varphi$ and $f_l^k$,
$\sum_{l\geq 1} \sum_{|\alpha|=m}| B(l,\alpha)| d\<M^l\>_t  \leq C \|Z(t)\|_m^2 dV_t$ for every $t\in [0,T]$.
Therefore, we deduce that we can find a  constant $K>0$ such that   $dT_3(t)\leq K \|Z(t)\|_m^2 dV_t$.  The above inequalities and \eqref{decomp1}
complete the proof of Condition {\bf (C1)}.

Since $L_r$ are linear operators, Condition {\bf (C1)} implies the following classical Monotonicity, Coercivity and
Hemicontinuity: for every   $Z, \zeta \in H^{m+1}$ and $L_{r,\lambda}$ defined by \eqref{Lrl},
\begin{align*}
&2\sum_{r=0}^{d_1}  \< L_r Z-L_r \zeta  \, ,\, Z-\zeta \>_m  dV^r_t + \sum_{l\geq 1}
\|S_l(Z)-S_l(\zeta)\|_m^2 d\<M^l\>_t \leq K \|Z-\zeta\|_m^2   dV_t , \\
&2\!\sum_{r=0}^{d_1}  \< L_{r,\lambda} Z  \, ,\, Z \>_m  dV^r_t + \sum_{l\geq 1}
\|S_l(Z)\|_m^2 d\<M^l\>_t +2\lambda \|Z\|_{m+1}^2\! \sum_{r=0}^{d_1}\!\! \!\! dV^r_{t} \leq K \|Z\|_m^2   dV_t
\end{align*}
a.s. in the sense of measures,  and for $Z_i\in H^{m+1}$, $i=1,2,3$, $r=0, \cdots, d_1$ and $\lambda >0$, the  map
$ a\in \CC \to \< L_{r,\lambda} (Z_1+aZ_2 \, ,\, Z_3\>_m$ is continuous.

The following condition {\bf (C2)} gathers some useful  
bounds on the operators $L_r$ and $S_l$ for $0\leq r\leq d_1$ and  $l\geq 1$.
\smallskip

\noindent {\bf Condition (C2)} There exist positive
  constants $K_i, i=2,3,4$ such that for $Z\in H^{m+1}$, $\lambda\in [0,1]$, $r=0, \cdots,
d_1$ and $l\geq 1$:
\begin{align}
& 2 \|L_{r,\lambda} Z\|_{m-1}^2 + \|S_lZ\|_m^2 \leq K_2 \|Z\|_{m+1}^2 \; \mbox{\rm  a.s.} \label{linear} \\
&
 |(S_l Z  ,Z)_m|\leq K_3 \|Z\|^2_m \; \mbox {\rm and } |(S_l Z ,G_l)_m|\leq K_4 \|Z\|_m \|G_l\|_{m+1}
 \; \mbox{\rm  a.s.} . \label{majoSl}
\end{align}
The inequality  \eqref{linear} is a straightforward consequence of the Cauchy-Schwarz inequality
 and of Assumption {\bf (A3(m))}.
Using integration by parts and Assumptions {\bf (A3(m))-(A4(m,p))}, we deduce that if $G_l(t)=G_{l,1}(t)+iG_{l,2}(t)$,
\begin{align*}
|(S_l Z,Z)_m|\leq& \frac{1}{2} \sum_{|\alpha|=m} \sum_{j=1}^{d}
\big|  \big[ \big( D_j \sigma_l^j D^\alpha X , D^\alpha X \big)
+  \big( D_j \sigma_l^j D^\alpha Y , D^\alpha Y \big)\big]\big| \\
& + \sum_{|\alpha|\vee|\beta|\leq m} \big| \big( C(\alpha, \beta, l) D^\alpha Z , D^\beta Z\big)\big|
\end{align*}
\begin{align*}
|(S_l Z,G_l)_m|\leq& \sum_{j=1}^{d}
\sum_{|\alpha|=m} \big[ \big| \big( \sigma_l^j  D^\alpha X , D_j D^\alpha G_{l,1} \big)\big|
+ \big| \big( \sigma_l^j D^\alpha Y , D_j D^\alpha G_{l,2}\big)\big|\\
& +  \sum_{|\alpha|\vee|\beta|\leq m} \big| \big( C(\alpha, \beta, l) D^\alpha Z , D^\beta G_l\big) )\big| ,
\end{align*}
for constants  $C(\alpha,\beta,l)$, $\alpha,\beta,l$ such that $\sup_{\alpha,\beta,l}C(\alpha,\beta,l)\leq C<\infty$.
 Hence a simple application of the Cauchy-Schwarz and Young inequalities implies inequality \eqref{majoSl}.

We then proceed as in the proof of Theorem~3.1 in \cite{KR-82} for fixed $\lambda >0$ (see also \cite{Pa} and \cite{GK}).
To ease notations, we do not write the Galerkin approximation as  the following estimates would be valid
 with constants which do not
depend on the dimension of the Galerkin approximation, and hence would still be true for the weak and weak$^\ast$ limit in
$L^2([0,T]\times \Omega ; H_{m+1})$ and $ L^2(\Omega ; L^\infty(0,T ; H_m))$.
Let us fix  a real number $N>0$ and let $\tau_N=\inf\{t\geq 0 : \|Z^\lambda(t)\|_m\geq N\}\wedge T$.
The It\^o formula, the stochastic parabolicity condition {\bf (C1)} and  the Davies inequality
imply that for any $t\in [0,T]$ and $\lambda\in (0,1]$,
\begin{align*}
& \EE\Big(\sup_{s\in [0,t]} \|Z^\lambda(s \wedge \tau_N)\|_m^2\Big) + 2\lambda
\EE\int_0^{t\wedge \tau_N} \|Z^\lambda(s )\|_{m+1}^2 ds
 \leq \EE \|Z_0\|_m^2\\
& \quad + 2 \sum_{r=0}^{d_1} \EE\int_0^{t} \big| \< F_r(s\wedge \tau_N) , Z^\lambda(s\wedge \tau_N) \>_m\big|  dV^r_s \\
&\quad + \sum_{l\geq 1} \EE\int_0^t \big[ 2\big|\big(S_l Z^\lambda(s\wedge\tau_N) , G_l(s\wedge\tau_N)\big)_m\big|
+ \|G_l(s\wedge \tau_N)\|_m^2  \big] d\<M^l\>_s \\
&\quad  + 6\; \EE\Big(\sum_{l\geq 1}  \Big\{ \int_0^t \big( S_l Z^\lambda(s\wedge \tau_N ) + G_l(s\wedge \tau_N) ,
 Z^\lambda (s\wedge \tau_N )\big)_m^2
d \<M^l\>_s\Big\}^{\frac{1}{2}}\Big)
\end{align*}
The Cauchy-Schwarz inequality, the upper estimate \eqref{majoSl} in  Condition {\bf (C2)}
 and inequalities \eqref{V1} - \eqref{V} in Assumption {\bf (A1)}
 imply the existence of some constant $K>0$ such that for any $\delta >0$,
\begin{align*}
& \EE\Big(\sup_{s\in [0,t]} \|Z^\lambda(s \wedge \tau_N)\|_m^2\Big) + 2\lambda
\EE\int_0^{t\wedge \tau_N} \|Z^\lambda(s )\|_{m+1}^2 ds
 \leq \EE \|Z_0\|_m^2\\
&\quad +  3 \; \delta  \EE\Big(\sup_{s\in [0,t]} \|Z^\lambda(s \wedge \tau_N)\|_m^2\Big)
+ {\tilde{K}}{\delta}^{-1} \sum_{r=0}^{d_1}  \EE\int_0^{t }  \|F_r(s)\|_m^2  dV^r_s \\
&\quad + \EE\int_0^t   \sum_{l\geq 1} K \big[ {\delta}^{-1} + 1\big]
 \|G_l(s)\|_{m+1}^2 
 d\<M^l\>_s  + \EE  \int_0^t \|Z^\lambda(s\wedge \tau_N)\|_m^2 dV_s.
\end{align*}
For $\delta =\frac{1}{6}$,  the Gronwall Lemma implies  that for some constant $C$ we have for all $N>0$ and $\lambda\in (0,1]$,
\[  \EE\Big(\sup_{s\in [0,t]} \|Z^\lambda(s \wedge \tau_N)\|_m^2\Big) + \lambda \EE\int_0^{t\wedge \tau_N}
 \|Z^\lambda (s )\|_{m+1}^2 ds
\leq C \EE \big(  \|Z_0\|_m^2 + K_m(T)\big).\]
As $N\to\infty$, we deduce that $\tau_N\to\infty$ a.s. and by the monotone convergence theorem,
\[  \EE\Big(\sup_{s\in [0,T]} \|Z^\lambda(s)\|_m^2\Big) + \lambda \EE\int_0^{T} \|Z^\lambda (s )\|_{m+1}^2 ds
\leq C \EE \big(  \|Z_0\|_m^2 + K_m(T)\big).\]
Furthermore, $Z^\lambda$ belongs a.s. to ${\mathcal C}([0,T], H^m)$. As in \cite{KR-82}, we deduce the existence of a sequence
$\lambda_n\to 0$
such that  $Z^{\lambda_n} \to Z$ weakly in $L^2([0,T]\times \Omega ; H^m)$.
 Furthermore, $Z$ is a solution to \eqref{Z} and \eqref{CIZ}
such that \eqref{apriori1} holds and is a.s. weakly continuous from $[0,T]$ to $H^m$.

The uniqueness of the solution follows from the growth condition \eqref{linear} in  {\bf (C2)} and the
monotonicity condition which is deduced from the stochastic parabolicity property  {\bf (C1)}.

(ii)  Suppose that Assumption {\bf (A4(m,p))} holds for $p\in [2,\infty)$. Set  $p=2 \tilde{p}$ with
$\tilde{p}\in [1,\infty)$;  the It\^o formula,
the stochastic parabolicity condition {\bf (C1)}, the growth conditions {\bf (C2)},
 the Burkholder-Davies-Gundy and Schwarz inequalities
yield  the existence of some constant $C_p$ which also depends in the various constants in Assumptions
{\bf (A1)-(A4(m,p))}, and conditions {\bf (C1)-(C2)}, such that:
\begin{align*}
& \EE\Big(\! \sup_{s\in [0,t]} \|Z(s)\|_m^{p} \Big) \leq  C_p \Big[\! \EE \|Z_0\|_m^{p}
 +  \EE\Big(\! \sup_{s\in [0,t]} \|Z(s)\|_m^{\tilde{p}} \Big| \sum_{r=0}^{d_1}\!
\int_0^t \!\|F_r(s)\|_m dV^r_s \Big|^{\tilde{p}}\Big) \\
&\quad + \EE \Big(\!  \sup_{s\in [0,t]} \|Z(s)\|_m^{\tilde{p}}
 \Big| \sum_{l\geq 1}\! \int_0^t \!\|G_l(s)\|_{m+1} d\<M^l\>_s \Big|^{\tilde{p}} \Big) \\
&\quad + \EE \Big( \! \sup_{s\in [0,t]} \|Z(s)\|_m^{\tilde{p}}
 \Big| \sum_{l\geq 1} \!\int_0^t \!\big[ \|Z(s)\|_m^2 + \|G_l(s)\|_m^2 \big] d\<M^l\>_s
\Big|^{\tilde{p}/2}  \Big) .
\end{align*}
Using the H\"older and Young inequalities, \eqref{bndpini}  as well as Assumptions {\bf (A1)} we deduce
the existence of a constant $K>0$ such that for any $\delta >0$
 \begin{align*}
& \EE\Big(\sup_{s\in [0,t]} \|Z(s)\|_m^{p} \Big) \leq  3\delta  \EE\Big(\sup_{s\in [0,t]} \|Z(s)\|_m^{p} \Big)
+ C(p) \Big[ \EE \|Z_0\|_m^{p} \\
&\quad + {K}^{\tilde{p}} \delta^{-1} \EE\Big| \int_0^t \sum_{r=0}^{d_1} \|F_r(s)\|_m^2 dV^r_s\Big|^{\tilde{p}}
+  {K}^{\tilde{p}} \delta^{-1} \EE\Big| \int_0^t \sum_{l\geq 1} \|G_l(s)\|_{m+1}^2 d\<M^l\>_s \Big|^{\tilde{p}}\\
&\quad +  {K}^{\tilde{p}-1} \delta^{-1} \EE\int_0^t \|Z(s)\|_m^{p} dV_s
+ \delta^{-1} \Big|\EE\int_0^t \sum_{l\geq 1} \|G_l(s)\|_m^2 d\<M^l\>_s \Big|^{\tilde{p}} \Big] .
\end{align*}
Let $\delta = \frac{1}{6}$
 and introduce the stopping time $\tau_N=\inf\{ t\geq 0 : \|Z(t)\|_m\geq N\}\wedge T$. Replacing $t$ by $t\wedge \tau_N$
in the above upper estimates, the Gronwall Lemma and \eqref{bndpini} prove that there exists a constant $C$ such that
$ \EE \Big( \sup_{s\in [0,t]\wedge \tau_N} \|Z(s)\|_m^{2p} \Big)
 \leq C \EE \big( \|Z_0\|_m^{p} + K_m(T)^{p/2}\big)$ for every $N>0$.
As $N\to \infty$  the monotone convergence theorem concludes the proof of \eqref{apriori2}.
This ends of the proof of Theorem \ref{Hmestimates}.
\end{proof}
\begin{rk} 
If $a^{j,k}_r(t,x)=0$, for example for the Schr\"odinger equation, Assumption {\bf (A2)} implies that $\sigma_l^j= 0$.
\end{rk} 

\subsection{Further a priori estimates on the difference}
Theorem \ref{Hmestimates}  is used to upper
estimate moments of the difference of two processes solutions to equations of type \eqref{Z}. For $\ee=0,1$,
$r=0, \cdots,\,  d_1$, $l\geq 1$, $j,k=1,\cdots,d$,  let $a^{j,k}_{\ee,r}(t,x), b^{j,k}_{\ee,r}(t), a^j_{\ee,r}(t,x)$,
$ a_{\ee,r}(t,x), b_{\ee,r}(t,x),
\s_{\ee,l}^j(t,x),  \s_{\ee,l}(t,x), \tau_{\ee,l}(t,x)$   be coefficients,  $F_{\ee,r}(t,x)$,
$G_{\ee,l}(t,x)$ be processes, and  let $Z_{\ee,0}$ be random variables  which satisfy
the assumptions {\bf (A1)}- {\bf (A3(m))} and  {\bf (A4(m,p))} for some $m\geq 1$, $p\in [2,\infty)$,  the
same martingales  $(M_{t}^l, t\in [0,T])$ and increasing processes $(V^r_{t}, t\in [0,T])$.
Let $L_{\ee,r}$ and $S_{\ee,l}$ be defined as in \eqref{Lr} and \eqref{Sl} respectively.
Extend these above  coefficients, operators,  processes and random variables to $\ee\in [0,1]$ as follows: if $f_0$ and
$f_1$ are given, for $\ee\in [0,1]$, let $f_\ee=\ee f_1 + (1-\ee) f_0$. Note that by convexity, all the previous assumptions
are satisfied for any $\ee\in [0,1]$. Given $\ee\in [0,1]$, let $Z_\ee$ denote the solution to the evolution
equation:  $Z_\ee(0,x)=Z_{\ee,0}(x)$ and
\begin{equation} \label{Zee}
dZ_\ee(t,x)=\sum_{r=0}^{d_1} \!\!\! \big[ L_{\ee,r} Z_\ee(t,x) + F_{\ee,r}(t,x)]\, dV^r_t
 + \sum_{l\geq 1} \big[ S_{\ee,l} Z_\ee(t,x) + G_{\ee,l}(t,x)\big] dM^l_t.
\end{equation}
Thus, Theorem \ref{Hmestimates} immediatly yields the following
\begin{corollary} \label{corZee}
With the notations above, the solution $Z_\ee$ to \eqref{Zee} with the initial condition $Z_{\ee,0}$ exists
and is unique with trajectories in $C([0,T];H^{m-1}) \cap L^\infty(0,T ; H^m)$. Furthermore,
the trajectories of $Z_\ee$ belong a.s. to
$C_{\textrm{w}}([0,T];H^{m})$ and
there exists a constant $C_p>0$ such that
\begin{equation}\label{aprioriee}
\sup_{\ee\in [0,1]}\, \EE\Big(\sup_{t\in [0,T]} \|Z_\ee(t,\cdot)\|_m^p \Big)
 \leq C_p\, \sup_{\ee\in \{0,1\}} \EE\Big(\|Z_{\ee,0}\|^p_m + K_m(T)^{p/2}\Big)<\infty .
\end{equation}
\end{corollary}
Following the arguments in \cite{GK}, this enables us to estimate moments of $Z_1-Z_0$ in terms of a process
$\zeta_\ee$ which is a formal derivative of $Z_\ee$ with respect to $\ee$.
Given operators or processes $f_\ee$, $\ee\in \{0,1\}$, set  $f'=f_1-f_0$.
\begin{theorem}\label{diff1}
Let $m\geq 3$, and $p\in [2,\infty)$. Then for any integer $\kappa =0, \cdots, m-2$
\begin{equation} \label{majodiffv}
\EE\Big( \sup_{t\in [0,T]} \|Z_1(t)-Z_0(t)\|_{\kappa}^p \Big)
\leq \sup_{\ee\in [0,1]} \EE\Big(\sup_{t\in [0,T]} \|\zeta_\ee(t)\|_\kappa^p\Big) ,
\end{equation}
where $\zeta_\ee$ is the unique solution to the following linear  evolution equation:
\begin{align} \label{zetaee}
d\zeta_\ee(t) = & \sum_{r=0}^{d_1} \big( L_{\ee,r}  \zeta_\ee(t,x) + L'_r Z_\ee(t,x) + F'_r(t,x)\big) dV_r(t)
\nonumber \\
& + \sum_{l\geq 1} \big( S_{\ee,l} \zeta_\ee(t,x) + S'_l Z_\ee(t,x) + G'_l(t,x)\big) dM_t^l\, ,
\end{align}
with the initial condition $Z'_0=Z_1-Z_0$. Furthermore,
\begin{equation}\label{apriorizetaee}
\sup_{\ee\in [0,1]} \EE\Big( \sup_{t\in [0,T]} \|\zeta_\ee(t)\|_{m-2}^p\Big) <\infty.
\end{equation}
\end{theorem}
\begin{proof}  
Using \eqref{aprioriee} we deduce that the processes $\tilde{F}_r(t,x)=L'_r Z_\ee(t,x) + F'_r(t,x)$ and
$\tilde{S}_l(t,x)=S'_l Z_\ee(t,x) + G'_l(t,x)$ satisfy the assumption  {\bf (A4(m-2,p))} with $m-2\geq 1$. Hence the existence
and uniqueness of the process $\zeta_\ee$, solution to \eqref{zetaee}, as well as \eqref{apriorizetaee}
 can be deduced from Theorem \ref{Hmestimates}.

We now prove \eqref{majodiffv} for $\kappa \in \{0, \cdots, m-2\}$ and assume that the right hand-side is finite.
Given  $(f_\ee , \ee\in [0,1))$, for  and $h>0$ and  $\ee\in [0,1]$ such that
 $\ee+h\in [0,1]$, set $\delta_h f_\ee=(f_{\ee+h}-f_\ee)/h$.
We at first prove that \eqref{majodiffv} can be deduced  from the following: for every $\ee\in [0,1)$, as $h\to 0$
is such that $h+\ee\in[0,1]$,
\begin{equation}\label{delta}
\EE\Big( \sup_{t\in [0,T]} \| \delta_h Z_\ee(t) - \zeta_\ee(t)\|_0^p\Big) \to 0.
\end{equation}
Indeed,  assume that \eqref{delta} holds and for $n>0$   let
$R_n=n^\kappa \Delta^\kappa (n\mbox{\rm Id} - \Delta)^{-\kappa} $  denote the $kapa$-fold composition of the
resolvent  of the Laplace operator $\Delta$ on the space $L^2=H^0$.
Then, by some classical estimates,  there exists a constant
$C(\kappa)>0$  such that for any $\phi \in L^2$, $\|R_n h\|_\kappa \leq C(\kappa) \|\phi\|_0$.
Hence \eqref{delta} yields that for every $n>0$, as $h\to 0$ with $\ee+h\in [0,1]$, we have
 $\EE\Big( \sup_{t\in [0,T]} \|\delta_h R_n Z_\ee(t) - R_n \zeta_\ee(t)\|_\kappa^p\Big) \to 0$.
 Furthermore, since for every integer $N\geq 1$, we  have
$Z_1-Z_0=\frac{1}{N}\sum_{k=0}^{N-1} \delta_{1/N} Z_{k/N} \leq \sup_{\ee\in [0,1]}
\delta_{1/N} Z_\ee$,   we deduce that for every $n>0$ and $p\in [ 2,\infty)$:
\[ \EE\Big( \sup_{t\in [0,T]} \|R_n Z_0(t) - R_n Z_1(t)\|_\kappa^p\Big) \leq \sup_{\ee\in [0,1]} \EE \Big( \sup_{t\in [0,T]}
\| R_n \zeta_\ee(t)\|_\kappa^p\Big).\]
Finally, if $\phi \in H^0$ is such that $\liminf_{n\to \infty} \|R_n\phi\|_\kappa =N_\kappa <\infty$, then $\phi \in H^\kappa$
and $\|\phi\|_\kappa \leq N_\kappa$. Thus, by applying the Fatou Lemma and using estimate \eqref{apriorizetaee} we can conclude the proof of \eqref{majodiffv}.

We will now prove the convergence \eqref{delta}. It is easy to see that the process $\eta_{\ee,h}(t,\cdot):=\delta_h Z_\ee(t,\cdot) - \zeta_\ee(t,\cdot)$
has initial condition $\eta_{\ee,h}(0)=0$, and is a solution of the evolution equation:
\begin{align*}
d\eta_{\ee,h}(t)=& \sum_{r=0}^{\;d_1} \big[ L_{\ee,r} \eta_{\ee,h}(t,\cdot) + L'_r \big(Z_{\ee +h}(t,\cdot) - Z_\ee(t,\cdot)\big) \big] dV^r_t\\
&+ \sum_{l\geq 1} \big[ S_{\ee,l} \eta_{\ee,h}(t,\cdot) + S'_l\big( Z_{\ee +h}(t,\cdot) - Z_\ee(t,\cdot)\big) \big] dM^l_t.
\end{align*}
Hence, using once more Theorem \ref{Hmestimates}, we deduce the existence of a constant $C_p>0$ independent of $\ee\in [0,1)$
 and $h>0$, such that  $\ee+h \in [0,1]$,
 \[
\EE\Big( \sup_{t\in [0,T]}  \|\delta_h Z_\ee(t) - \zeta_\ee(t)\|_0^p  \Big)
\leq C_p \EE\Big( \int_0^T \|Z_{\ee +h}(t) - Z_{\ee}(t)\|_2^2\,  dV_t\Big)^{p/2}.
\]
Using the  interpolation  inequality $\|\phi \|_2 \leq C \|\phi\|_0^{1/3} \|\phi\|_3^{2/3}$, see for instance
Proposition 2.3 in \cite{Lions+M1972}, the H\"older inequality
and the estimate \eqref{aprioriee} with $m=3$ from Corollary \ref{corZee}, we deduce that 
\[
\EE\Big( \sup_{t\in [0,T]} \|\delta_h Z_\ee(t) - \zeta_\ee(t)\|_0^p\Big)
\leq C \Big[ \EE\Big( \sup_{t\in [0,T]} \|Z_{\ee+h}(t) - Z_\ee(t)\|_0^p\Big) \Big]^{1/3}. \]
Finally, the process $\Phi_{\ee,h}(t,\cdot) = Z_{\ee+h}(t,\cdot) - Z_\ee(t,\cdot)$ is solution to  the evolution equation
\begin{align*}
d\Phi_{\ee,h}(t)=& \sum_{r=0}^{d_1} \big[ L_{\ee,r} \Phi_{\ee,h}(t,\cdot)
 + h  L'_r Z_{\ee +h}(t,\cdot) +h F'_r(t,\cdot) \big] dV^r_t\\
&+ \sum_{l\geq 1} \big[ S_{\ee,l} \Phi_{\ee,h}(t,\cdot) + h  S'_l Z_{\ee +h}(t,\cdot) + h G'_l(t,\cdot)\big) \big] dM^l_t,
\end{align*}
with the initial condition $\Phi_{\ee,h}(0)=h(Z_1-Z_0)$.  Thus,  \eqref{aprioriee} and Theorem  \ref{Hmestimates}
prove the existence of a constant $C$, which does not depend on $\ee\in [0,1)$ and $h>0$ with $\ee+h\in [0,1]$,  and such that
\[ \EE\Big( \sup_{t\in [0,T]} \|Z_{\ee+h}(t) - Z_\ee(t)\|_0^p\Big) \leq C  h^{p/3}. \]
This concludes the proof of \eqref{delta} and hence that of the Theorem \ref{diff1}.
\qed
\end{proof}

\section{Speed of convergence} \label{speedabstract}
\subsection{Convergence for time-independent coefficients} \label{speedtimeinde}
For  $r=0, \cdots, d_1$, $\ee=0,1$, let $(V_{\ee,t}^r, t\in [0,T])$ be increasing processes which satisfy
 Assumptions {\bf (A1)}, {\bf (A2)},  {\bf (A3(m+3))} and {\bf (A4(m+3,p))}
 for some integer $m\geq 1$, some $p\in[2,+\infty)$ separately for the increasing processes
$(V^r_{\ee,t}, t\in [0,T])$, the same increasing process $(V_t, t\in [0,T])$  and the initial conditions $Z_{\ee,0}$,  $\ee =0,1$.
 For $\ee=0,1$, let $Z_\ee$ denote the  solution to the evolution equation
\begin{equation} \label{Z01}
dZ_\ee(t,x)=\sum_{0\leq r \leq d_1} \big[ L_r Z_\ee(t,x) + F_r(x)]\, dV_{\ee,t}^r
 + \sum_{l\geq 1} \big[ S_l Z_\ee(t,x) + G_l(x)\big] dM^l_t,
\end{equation}
with the initial conditions $Z_0(0,\cdot)=Z_{0,0}$  and   $Z_1(0,\cdot)=Z_{1,0}$ respectively.
Let
\[ A:= \sup_{\omega\in \Omega} \, \sup_{t\in [0,T]}\, \max_{r=0,1, \cdots, d_1} |V^{r}_{1,t} - V^{r}_{0,t}|.\]
Then the $H^m$ norm of the difference $Z_1-Z_0$ can be estimated in terms of $A$ as follows
 when the coefficients of $L_{r}$   and $F_{r} $ are time-independent.
Indeed, unlike the statements in \cite{GM}, but as it is clear from the proof, the diffusion coefficients
$\sigma_{l}$ and $G_{l}$ can depend on time. 
 \begin{theorem} \label{Thsplitting1}
Let 
$L_r$ and $F_r$ be 
time-independent, $\FF_0$-measurable,
$V^{r}_{\ee}$, $\ee=0,1$, $M_l$ be as above and let Assumptions {\bf (A1)},
{\bf (A2)}, {\bf (A3(m+3))}
and {\bf (A4(m+3,p))} be satisfied for some $m\geq 0$ and some $p\in [2,+\infty)$. Suppose  furthermore
that
\begin{equation} \label{A6}
\EE \Big( \Big| \sum_{r=0}^{d_1} \|F_r\|_{m+1}^2\Big|^{p/2}  +
\sup_{s\in [0,T]} \Big| \sum_{l\geq 1} \|G_r(s)\|_{m+2}^2\Big|^{p/2}    \Big) <\infty.
\end{equation}
 Then there exists a constant $C>0$,
which only depends on $d$ and the constants in the above assumptions,  such that  the solutions
$Z_0$ and $Z_1$ to \eqref{Z01} satisfy the following inequality:
\[  
\EE\Big( \sup_{t\in [0,T]} \|Z_1(t)-Z_0(t)\|_m^p\Big) \leq C\, \Big( \EE(\|Z_{1,0} -Z_{0,0}\|^p_m) + A^p \Big).
\]  
\end{theorem}
The proof of Theorem \ref{Thsplitting1} will require several steps. Some of them do not  depend on the fact that
the coefficients are time independent;  we are keeping general coefficients whenever this is possible.
 The first step is to use Theorem \ref{diff1}
and hence to define a process $Z_\ee$ for any $\ee\in [0,1]$; it does not depend on the fact that the coefficients
are time-independent and extends to the setting of the previous section. For $\ee\in [0,1]$, $r=0, \cdots d_1$,  $t\in [0,T]$
and $x\in \RR^d$, let
\[ V^r_{\ee,t} = \ee V^r_{1,t} + (1-\ee) V^r_{0,t} , \quad  \rho_{\ee,t}^r = dV^r_{\ee,t}/dV_t\]
and  for $j,k=1,\cdots d$,   set $a_{\ee,r}^{j,k}(t,x) =  \rho_{\ee,t}^r a_r^{j,k}(t,x)$,
 $b_{\ee,r}^{j,k}(t)= \rho_{\ee,t}^r
b_r^{j,k}(t)$, $a_{\ee,r}^j (t,x) =  \rho_{\ee,t}^r a_r^j(t,x)$, $a_{\ee,r}(t,x)=  \rho_{\ee,t}^r a_r(t,x)$,
 $b_{\ee,r}(t,x)=  \rho_{\ee,t}^r b_r(t,x)$, $L_{\ee,r}=  \rho_{\ee,t}^r L_r$, $F_{\ee,r}(t,x) =  \rho_{\ee,t}^r F_r(t,x)$.
Then for $\ee \in [0,1]$, the solution $Z_\ee(t,\cdot)$  to equation \eqref{Z} with the increasing processes
$V^r_{\ee,t}$ can be rewritten
as  \eqref{Zee} with the initial data $Z_\ee(0)=\ee Z_{1,0} + (1-\ee) Z_{0,0}$ and the operators (resp. processes)
$S_{\ee,l} = S_l$ (resp. $G_{\ee,l}= G_l$). Furthermore, we have
\begin{align*}     \sum_{r=0}^{d_1}  \sum_{j,k=1}^{d} & \lambda^j \lambda^k \big( a^{j,k}_{\ee,r}(t,x)
+i b^{j,k}_\ee(t)\big) dV^r_t = \\
& \sum_{0\leq r \leq d_1} \sum_{j,k=1}^{d}
 \lambda^j \lambda^k \big(a^{j,k}_{r}(t,x)+ ib^{j,k}_r(t)\big) dV^r_{\ee,t}.
\end{align*}
Hence the conditions {\bf (A1), (A2)}, {\bf (A3(m+3))} and  {\bf (A4(m+3,p))} are satisfied. Therefore, using Theorem
\ref{diff1}, one deduces that the proof of Theorem \ref{Thsplitting1} reduces to check that
\begin{equation} \label{majoder1}
\sup_{\ee\in [0,1]} \EE\Big( \sup_{t\in [0,1]} \|\zeta_\ee(t)\|_m^p\Big)
 \leq C \big( \EE\|Z_{1,0}-Z_{0,0}\|_m^p + A^p\big),
\end{equation}
where if one lets  $A^r_t=V^r_{1,t} - V^r_{0,t}$, the process  $\zeta_\ee$ is   the unique solution to \eqref{zetaee}
which here can be written as follows: for $t\in [0,T]$ one has
\begin{eqnarray} \label{zetaeebis}
d\zeta_\ee(t) &=&  \sum_{r=0}^{d_1}  \big[ L_r \zeta_\ee(t,x) dV^r_{\ee,t}
 + \big( L_r Z_\ee(t,x) + F_{r}(t,x)\big)
 dA^r_t \big]\nonumber \\
&&  + \sum_{l\geq 1}  S_l \zeta_\ee(t,x) dM_l(t),
\end{eqnarray}
and the initial condition  is $\zeta_\ee(0)=Z_{1,0}-Z_{0,0}$.

\noindent To ease notations, given a multi-index $\alpha$, $j,k\in \{1, \cdots , d\}$ and $Z$ smooth enough,  set
$ Z_\alpha = D^\alpha Z, \quad Z_{\alpha,j}=D^\alpha D_j Z \quad \mbox{\rm and }\,  Z_{\alpha,j,k}=D^\alpha
D_j D_k Z$,
so that for $Z,\zeta \in H^m$, $(Z,\zeta)_m = \sum_{|\alpha \leq m} (Z_\alpha,\zeta_\alpha)_0$. Let
\[ {\mathcal A} =
\Big\{ \sum_{\alpha} \sum_{\beta} a^{\alpha,\beta } Z_\alpha Z_\beta \; ;\:  a^{\alpha,\beta} \;
\mbox{\rm uniformly bounded and complex-valued},\;  Z \in H^{m+3}\Big\}
\]
and for $\Phi, \Psi \in {\mathcal A} $ set $\Phi \sim \Psi$ if there exists $Z\in H^m$ such that
$\int_{\RR^d} (\Phi - \Psi)(x) dx = \int_{\RR^d} \Gamma(x) dx$, where $\Gamma$ is a function defined by
\begin{equation} \label{defPalpha}
\Gamma(x) =\sum_{|\alpha| \leq m } Z_\alpha (x)  \overline{ P^\alpha Z (x)} \quad \mbox{\rm with }\;
  P^\alpha Z = \sum_{|\beta|\leq m} \gamma^{\alpha,\beta}
Z_\beta,
\end{equation}
for some complex-valued functions $\gamma^{\alpha,\beta}$ such that  $|\gamma^{\alpha, \beta}|$ 
are  estimated from above by the constants appearing in Assumptions {\bf (A1), (A2)}, {\bf (A3(m+3))}, {\bf (A4(m+3,p))}.
Note that if $\Gamma$ is as  above, then for some constant $C_m(\Gamma)$ we have
\begin{equation} \label{majogamma} \int_{\RR^d} |\Gamma(x)|\, dx
\leq C_m(\Gamma) \|Z\|_m^2.
\end{equation}
For $\ee >0$, $j,k=1, \cdots d$, $l\geq 1$ and $t\in [0,T]$, set $q_t^l = d\langle M^l\rangle_t / dV_t$ and let
\[  \tilde{a}^{j,k}_\ee :=  \tilde{a}^{j,k}_\ee(t,\cdot) =
\sum_{r=0}^{d_1} a^{j,k}_{\ee,r}(t,\cdot) - \frac{1}{2}
\sum_{l\geq 1} \s_l^j(t,\cdot) \s_l^k(t,\cdot) \,  q^l_t. \]
For $m\geq 0$ and $z\in H^{m+1}$, set
\begin{equation}\label{qf}
[Z]_m^2 : = [Z]_m^2(t) = \sum_{j,k=1}^{d} \big( \tilde{a}^{j,k}_\ee(t) D_jZ\, ,\, D_k Z\big)_m + C_m \|Z\|_m^2,
\end{equation}
with $C_0=0$ and $C_m>0$ to be chosen later on, so that the right handside of \eqref{qf} is non negative.
Given $Z,\zeta \in H^{m+1}$, set  $[Z,\zeta]_m = \frac{1}{4} \big( [Z+\zeta]_m^2 +
[Z-\zeta]_m^2\big)$. We at first prove that $[.]_m$ defines  a non-negative  quadratic form on $H^{m+1}$
 for  some large enough constant $C_m$.
Once more, this result does not require that the coefficients  be
time-independent.
\begin{lemma}\label{Cm}
Suppose that the conditions {\bf (A1), (A2)} and  {\bf (A3(m+1))} are satisfied.
Then    there exists a  large enough  constant $C_m$ such that \eqref{qf}
defines a  non-negative quadratic form on $H^{m+1}$.
\end{lemma}
\begin{proof}
Assumption {\bf (A2)}  for $\ee\in \{0,1\}$
implies that \eqref{qf} holds for $m=0$ and $C_0=0$. Let $m\geq 1$ and
$\alpha$ be a multi-index such that $1\leq |\alpha|\leq m$. The Leibnitz formula implies the existence of constants
$C(\alpha,\beta,\gamma)$ such that  for $Z\in H^{m+1}$,
\begin{align*}   \Big( \sum_{j,k=1}^{d} &  \big( \tilde{a}^{j,k}_\ee(t)  D_jZ\big)_\alpha\, ,\,  Z_{\alpha,k}\Big)_0  = \\
& \sum_{j,k=1}^{d} \! \big( \tilde{a}^{j,k}_\ee(t) Z_{\alpha,j} \, ,\, Z_{\alpha,k} \big)_0
+ \!\! \sum_{\beta + \gamma = \alpha, |\beta |\geq 1}\!\!\! C(\alpha,\beta,\gamma) I_\ee^{\alpha,\beta,\gamma}(t) ,
\end{align*}
where
 $I^{\alpha,\beta,\gamma}_\ee(t) :=
\sum_{j,k=1}^{d} \big(  D^\beta \tilde{a}^{j,k}_\ee (t) Z_{\gamma,j}\, ,\,  Z_{\alpha, k} \big)_0 $.
Furthermore given $m\geq 1$,  multi-indices $\alpha$ with  $|\alpha |\leq m$ and
$Z\in H^{m+1}$,
using Assumption {\bf (A2)} we deduce that
$ \sum_{1\leq j,k \leq d }  \big( \tilde{a}^{j,k}_\ee(t) Z_{\alpha,j}\, ,\,  Z_{\alpha,k} \big)_0 \geq 0$.
Thus, the proof reduces to check that
\begin{equation} \label{aequiv}
  I^{\alpha,\beta,\gamma}_\ee(t)  \sim 0 .
\end{equation}
Indeed,  then the upper estimate \eqref{majogamma} proves  \eqref{qf}, which concludes the proof of the Lemma.
 Integration by parts implies $ I_\ee^{\alpha,\beta,\gamma}(t) =
-\sum_{j,k=1}^d  \Big( D_k\big(  D^\beta \tilde{a}_\ee^{j,k}(t) Z_{\gamma,j}\big),   Z_\alpha\Big)_0$.
 Since  $|\beta|\leq m$,   by Assumption {\bf (A3(m+1))}
 we know that $D_k D^\beta \tilde{a}^{j,k}_\ee(t)$ is bounded  and hence
  $  I_\ee^{\alpha,\beta,\gamma}(t)
\sim -\sum_{j,k=1}^d \!\! \big( D^\beta \tilde{a}^{j,k}_\ee(t) Z_{\gamma,j,k}\, ,\,  Z_\alpha\big)_0$.
If $|\beta |\geq 2$, then $|\gamma |\leq m-2$; hence
by {\bf (A3(m))} we deduce that   $  I_\ee^{\alpha,\beta,\gamma}(t) \sim 0$. If $|\beta| =1$, then $ \tilde{a}_\ee^{j,k}(t) =
 \tilde{a}_\ee^{k,j}(t)$, so that
\begin{align*}
  I_\ee^{\alpha,\beta,\gamma}(t) &  = \sum_{j,k=1}^{d} \big(  D^\beta  \tilde{a}_\ee^{j,k}(t) Z_{\gamma,j} \, ,\,
 D^\beta Z_{\gamma,k} \big)_0\\
& \sim \frac{1}{2} \sum_{j,k=1}^{d} \int_{\RR^d}  D^\beta
 \tilde{a}_\ee^{j,k}(t,x) D^\beta \big( X_{\gamma,j}(.) X_{\gamma,k}(.)+ Y_{\gamma,j}(.) Y_{\gamma,k}(.)\big)(x)
 dx   .
 \end{align*}
Thus, integrating by parts  and using {\bf (A3(2))}  and the inequality $|\gamma|+1 \leq m$, we deduce that
$I^{\alpha,\beta,\gamma}_\ee(t)  \sim - \, \frac{1}{2} \sum_{j,k=1}^{d} \big( D^\beta D^\beta \tilde{a}^{j,k}_\ee(t)
Z_{\gamma,j} \, ,\, Z_{\gamma,k}\big)_0\sim 0$. This concludes the proof.
\end{proof}
The following lemma gathers some technical results which again  hold for time-dependent coefficients.
\begin{lemma} \label{tech}
Suppose that the assumptions of Theorem \ref{Thsplitting1} hold. There exists a constant $C$ such that for $\zeta \in H^{m+1}$ one has $dV_t$ a.e.
\begin{equation} \label{p}
p(\zeta) := 2\sum_{0\leq r\leq d_1 } \rho_{\ee,t}^r \langle \zeta, L_r \zeta\rangle_m + \sum_{l\geq 1} q^l_t \|S_l \zeta\|_m^2 +
2 [\zeta]_m^2 \leq C \|\zeta\|_m^2.
\end{equation}
For any $\tilde{r}= 0,\,  1, \,  \cdots,\,  d_1$,  $Z\in H^{m+3}$ and  $\zeta \in H^{m+1}$ let
\[ q_{\tilde{r}}(\zeta,Z)  = \sum_{0\leq r\leq d_1 }   \rho_{\ee,t}^r
\big[ \langle L_r \zeta , L_{\tilde{r}} Z\rangle_m + \langle \zeta , L_{\tilde{r}} L_r Z\rangle_m\big]
+ \sum_{l\geq 1} q^l_t (S_l \zeta , L_{\tilde{r}} S_l Z)_m .\]
Then  there exists a constant $C$ such that for any $Z\in H^{m+3}$ and $\zeta\in H^{m+1}$, one has $dV_t$ a.e.
\begin{equation} \label{q}
|q_{\tilde{r}} (\zeta,Z)|\leq C \|Z\|_{m+3} \big( \|\zeta\|_m + [\zeta]_m \big).
\end{equation}
\end{lemma}
\begin{proof}
Suppose at first that $\zeta \in H^{m+2}$; since the upper estimates \eqref{p} and \eqref{q} only involve the
$H^{m+1}$-norm  of $\zeta$, they will follow by approximation. Then  we have
\[ \sum_{0\leq r\leq  d_1} 2  \rho_{\ee,t}^r\,  \langle \zeta, L_r \zeta\rangle_m
+ \sum_{l\geq 1}  q_t^l\,  \|S_l \zeta\|_m^2  =
\sum_{|\alpha|\leq m}  Q_t^\alpha(\zeta,\zeta) ,   \]
where
\[  Q_t^\alpha(\zeta,\zeta) = 2 \sum_{0\leq r\leq  d_1} \rho_{\ee,t}^r\;  \big( \zeta_\alpha \, ,\,
 (L_r \zeta)_\alpha\big)_0  +
\sum_{l\geq 1} q^l_t \; \| (S_l \zeta)_\alpha\|_0^2. \]
Integration by parts and assumption {\bf (A3(m))} imply that for $|\alpha|\leq m$, one has
\begin{align*}
2\big( \zeta_\alpha \, , \, (a^j_{\ee,r}  \zeta_j)_\alpha \big)_0
\sim 2 \int_{\RR^d} a^j_{\ee,r} (t,x)& \big( X_\alpha(x)  X_{\alpha,j}(x) + Y_\alpha(x) Y_{\alpha,j}(x)\big) dx \\
 = \int_{\RR^d}  a^j_{\ee,r} (t,x) \big( X_\alpha^2 + Y_\alpha^2\big)_j(x)\,  dx &
 \sim - \big( a^j_{\ee,r}(t)_j\,  \zeta_\alpha\, ,\,  \zeta_\alpha\big)_0 \sim 0, \\
\big( \zeta_\alpha \, ,\, \big( (a_{\ee,r} + ib_{\ee,r} ) \zeta)_\alpha\big)_0  & \sim 0, \\
2 \Big(  \big(\s_l^j  \zeta_j \big)_\alpha \, ,
 \, \big( (\s_l +i\tau_l)\zeta \big)_\alpha \Big)_0
& \sim   2 \Big(  \s_l^j \, \zeta_{\alpha,j} \, ,\,
(\s_l +i \tau_l) \zeta_\alpha \Big)_0   \\
& \sim - \int_{\RR^d} \big(  \s_l^j \s_l )  \big)_j(x)\,  |\zeta_\alpha(x)|^2\,  dx \sim 0, \\
\big\| \big( (\s_k + i \tau_k) \zeta)_\alpha \big\|_0^2
 & \sim 0.
\end{align*}
Finally,  we have $\Big( \zeta_\alpha \, ,\,
\sum_{j,k} \big(D_k\big( i b_r^{j,k}(t) D_j \zeta\big)_\alpha\Big)_0  =0$.
Set $L_r^0 \zeta = \sum_{j,k=1}^dD_k\big( a^{j,k}_r D_j \zeta\big) $ and $S_l^0 \zeta =
\sum_{j=1}^d \big( \s_l^j+i\tau_l^j\big) D_j \zeta$. Then we have
\begin{equation} \label{Qalpha}
Q^\alpha_t(\zeta,\zeta) \sim 2\sum_{r=0}^{d_1} \rho_{\ee,t}^r
\big( \zeta_\alpha \, ,\,  (L_r^0 \zeta )_\alpha \big)_0  +
\sum_{l\geq 1} q^l_t \|\big(  S^0_l \zeta\big)_\alpha\|^2_0.
\end{equation}
If $m =0$, integration by parts proves that the right hand side of \eqref{Qalpha} is equal  to $-2 [\zeta]_0^2$ (with $C_0=0$).
Let $m\geq 1$ and $\alpha$ be a multi index such that $m\geq |\alpha|\geq 1$;  set $\Gamma(\alpha) = \{ (\beta,\gamma) :
\alpha = \beta + \gamma , |\beta|=1\}$.
For $\phi, \psi \in H^m$, let $C(\beta,\gamma)$ be coefficients such that:
\[D^\alpha(\phi \psi)=\phi D^\alpha \psi + \sum_{(\beta,\gamma)\in \Gamma(\alpha)}
C(\beta,\gamma) D^\beta \phi \, D^\gamma \psi +
\sum_{\beta+\gamma=\alpha , |\beta|\geq 2} C(\beta,\gamma) D^\beta \phi \, D^\gamma \psi.\]
 This yields
\begin{align*}
\sum_{l\geq 1} q^l_t & \big\| \big(S_l^0 \zeta \big)_\alpha\big\|_0^2  \sim  \sum_{l\geq 1}  q^l_t \sum_{j,k=1}^{d} \Big\{
\big( \sigma_l^k \, \zeta_{\alpha,k} \, ,\,
 \s_l^j  \, \zeta_{\alpha,j} \big)_0 \\
& \quad + 2 \sum_{(\beta,\gamma) \in \Gamma(\alpha)}
 C(\beta,\gamma) \big( D^\beta \s_l^k  \zeta_{\gamma,k} \, ,\, \s_l^j
 \zeta_{\alpha,j}\big)_0 \Big\}.\\
 \sim & \sum_{l\geq 1}  q^l_t \sum_{j,k=1}^{d}
\Big\{ \Big( \s_l^k \s_l^j  \, \zeta_{\alpha,k} \, ,\, \zeta_{\alpha,j} \Big)_0
+ 2 C(\alpha,\beta) \, \Big( D^\beta \s_l^k \s_l^j \,
  \zeta_{\gamma,k}\, ,\, \zeta_{\alpha,j}\Big)_0 .
\end{align*}
Since for $(\beta,\gamma) \in \Gamma(\alpha)$ we have  $|\gamma| +1\leq |\alpha|\leq m$ while $|\beta|+1=2$,
 integrating by parts and  using
{\bf (A3(m))}  we have for fixed $l$,
\begin{align*}
2 q^l_t  \sum_{j,k} \Big(   D^\beta  \sigma_l^k  \s_l^j
 \zeta_{\gamma,k}\, ,\,  \zeta_{\alpha,j}  \Big)_0
 = - q_t^l \sum_{j,k} \big(  D^\beta (\sigma_l^k \sigma_l^j)
 \zeta_{\gamma,j,k} \, ,\, \zeta_\alpha\big)_0 .
\end{align*}
Furthermore, integration by parts  and {\bf (A3(m))}  yield
\begin{align*}   2\rho_{\ee,t}^r \big( \zeta_\alpha\, ,\,  (L^0 \zeta)_\alpha\big)_0
 \sim & - 2 \sum_{j,k}\Big\{ \big(  a^{j,k}_{\ee,r} \zeta_{\alpha,j}\, ,\,
\zeta_{\alpha,k}\big)_0 \\
& -   \sum_{(\beta,\gamma)\in \Gamma(\alpha)}
C(\beta,\gamma) \big( D^\beta(a^{j,k}_{\ee,r})  \zeta_{\gamma,j}\, ,\,  \zeta_{\alpha,k}\big)_0 \Big\}.
\end{align*}
Therefore,  the definition of $\tilde{a}^{j,k}_{\ee,r}$, \eqref{Qalpha}
and   \eqref{aequiv}  yield
\begin{align*}
Q^\alpha_t(\zeta,\zeta) &\sim -2 \sum_{j,k} \big(  \tilde{a}^{j,k}_{\ee}\, \zeta_{\alpha,j} \, ,\,
\zeta_{\alpha,k} \big)_0 - 2  \sum_{j,k} \sum_{(\beta,\gamma)\in \Gamma(\alpha)}
C(\beta,\gamma) \big( D^\beta(\tilde{a}^{j,k}_{\ee} ) \zeta_{\gamma,j}\, ,\,  \zeta_{\alpha,k} \big)_0\\
&\sim -2 \sum_{j,k} \Big( \big( \zeta_j \tilde{a}^{j,k}_{\ee}\big)_\alpha
\, ,\,  \zeta_{\alpha,k}\Big)_0  .
\end{align*}
Hence  for $\zeta \in H^{m+1}$,
\begin{equation}\label{eqp}
p(\zeta)=\sum_{|\alpha|\leq m} \int_{\RR^d} Q^\alpha_t (\zeta,\zeta) dx + 2 [\zeta]_m^2 = 2
\sum_{|\alpha|\leq m}   (\zeta_\alpha,P^\alpha \zeta)_0,
\end{equation}
for some operator $P^\alpha$ which satisfies \eqref{defPalpha}. Hence \eqref{majogamma}
 concludes the proof of \eqref{p}.

Polarizing  \eqref{eqp},  we deduce that for $\tilde{Z},\zeta\in H^{m+1}$,
\begin{align*}
\sum_{r=0}^{d_1} \rho_{\ee,t}^r \big[ \langle \tilde{Z}, L_r \zeta\rangle_m
  + \langle L_r \tilde{Z},\zeta \rangle_m\big] &  +
\sum_{l\geq 1}  q^l_t  \big( S_l \tilde{Z}, S_l \zeta )_m  + 2 [\tilde{Z}, \zeta]_m \\
&  = \sum_{|\alpha|\leq m} \big[ (\tilde{Z}_\alpha , P^\alpha\zeta)_0
+ (\zeta_\alpha,P^\alpha \tilde{Z})_0 \big].
\end{align*}
Let $\tilde{r}\in \{0,1, \cdots, d_1\}$ and for $Z\in H^{m+3}$, $\zeta\in H^{m+1}$, set $\tilde{Z} = L_{\tilde{r}} Z$;
 then if one sets
\begin{align*}
q_{\tilde{r}} (\zeta , Z): =
\sum_{r=0}^{d_1}  \rho_{\ee,t}^r
\big[ \langle L_r \zeta, L_{\tilde{r}} Z\rangle_m + \langle \zeta , L_{\tilde{r}} L_r Z \rangle_m\big]
+ \sum_{l\geq 1}   q^l_t\, \big( S_l \zeta , L_{\tilde{r}} S_l Z\big)_m ,
\end{align*}
one  deduces that
\begin{align*}
& q_{\tilde{r}}(\zeta, Z) \!  +\!  \sum_r \rho^r_{\ee,t} \big( \zeta, [L_r L_{\tilde{r}}\! -\!  L_{\tilde{r}} L_r] Z \big)_m
\!\! +\!  2 \sum_l q^l_t \big( S_l \zeta, [S_l L_{\tilde{r}} \! -\! 
 L_{\tilde{r}} S_l] Z \big)_m \! +  2 [\zeta, L_{\tilde{r}} Z]_m \\
& =  \sum_{|\alpha|\leq m}
\big[ (D^\alpha L_{\tilde{r}} Z , P^\alpha \zeta)_0  +  (D^\alpha \zeta , P^\alpha L_{\tilde{r}} Z )_0 \big] .
\end{align*}
The operators $L_r L_{\tilde{r}} - L_{\tilde{r}} L_r $ and $S_l L_{\tilde{r}} - L_{\tilde{r}} S_l$ are of order 3 and 2
respectively. Hence integration by parts and the Cauchy Schwarz inequality imply  that
\begin{align*}
|q_{\tilde{r}} (\zeta, Z) |\leq C \|\zeta\|_m \|Z\|_{m+3} + C [\zeta ]_m [Z]_m .
\end{align*}
Finally,   \eqref{qf} and  Assumption {\bf (A3(m))}  imply  that for $Z\in H^{m+1}$,
\[ [Z]_m^2 \leq C \|Z\|_{m+1}^2 + C_m \|Z\|_m^2 \leq C \|Z\|_{m+1}^2.\]
This concludes the proof of \eqref{q}.
 \end{proof}
The following lemma is based on some time integration by parts and requires the coefficients of $L_r$
and $F_r$ to be time independent.
\begin{lemma} \label{J-int}
Let the assumptions of Theorem \ref{Thsplitting1} be satisfied
and  $Z_\ee$ (resp. $\zeta_\ee$) denote  the processes   defined by \eqref{Z01} (resp. \eqref{zetaeebis}).
For $r= 0, \cdots, d_1$ and $t\in [0,T]$, let $A^r_t=V^r_{1,t} - V^r_{0,t}$ and set
\begin{equation} \label{Jeet}
 J_{\ee,t} := \sum_{r=0}^{d_1} \int_0^t \big( \zeta_\ee(s) , L_r Z_\ee (s) + F_r \big)_m dA^r_s.
 \end{equation}
Then there exists a constant $C$ such that for any stopping time $\tau \leq T$,
\begin{align}    \label{momentJ}
& \EE\Big[ \sup_{t\in[0, \tau]} \Big(J_{\ee,t} -\int_0^t [\zeta_\ee(s)]^2_m dV_s\Big)_+^{p/2}\Big] \nonumber  \\
&\qquad \qquad \leq \frac{1}{4p} \EE\Big( \sup_{t\in[0, \tau]} \|\zeta_\ee(s)\|_m^2 \Big)
+ C \Big( A^p + \EE\int_0^\tau \|\zeta_\ee(s)\|_m^p \, dV_s \Big) .
\end{align}
\end{lemma}
\begin{proof}
The main problem is to upper estimate $J_{\ee,t}$ in terms of $A$ and not in terms of the total variation of the
measures  $d A^r_t$.
This requires some  integration by parts; equations \eqref{Z01} and \eqref{zetaeebis} imply:
\begin{equation}\label{IPPJee}
 J_{\ee,t} = \sum_{r=0}^{d_1}
 \big( \zeta_\ee(t) \, ,\, L_r Z_\ee(t) + F_r \big)_m A^r_t - \sum_{1\leq k\leq 4} J^k_{\ee,t},
\end{equation}
where for $t\in [0,T] $  we set:
\begin{align*}
J^1_{\ee,t} = & \sum_{r} A^r_s \Big[ \sum_{\tilde{r}} \langle L_{\tilde{r}}\zeta_\ee(s) , L_r Z_\ee(s) + F_r \rangle_m
+ \langle \zeta_\ee(s) , L_r [L_{\tilde{r}} Z_\ee(s) + F_{\tilde r}] \rangle_m \Big] dV^r_{\ee,s}, \\
J^2_{\ee,t} = & \int_0^t \sum_r A^r_s \sum_{l\geq 1}
\big( S_l(s) \zeta_\ee(s) , L_r[S_l(s) Z_\ee(s) + G_l(s)] \big)_m d\langle M^l\rangle_s,\\
J^3_{\ee,t}=&  \int_0^t \sum_r A^r_s  \sum_{l\geq 1}\big[  \big( S_l(s) \zeta_\ee(s) , L_r Z_\ee(s) + F_r \big)_m \\
& \quad +  \big( \zeta_\ee(s) , L_r [S_l(s) Z_\ee(s) + G_l(s) ] \big)_m \big] dM^l_s,\\
J^4_{\ee,t}=& \int_0^t \sum_r A^r_s \Big[ \sum_{\tilde{r}}
\big( L_{\tilde{r}} Z_\ee(s) + F_{\tilde{r}} , L_r Z_\ee(s) + F_r \big)_m \Big] dA^{\tilde{r}}_s   .
\end{align*}
Note that
\[ J^4_{\ee,t}=\frac{1}{2} \int_0^t \sum_{r,\tilde{r}} \big( L_{\tilde{r}} Z_\ee(s) + F_{\tilde{r}} , L_r Z_\ee(s) + F_r \big)_m
d(A^r_s A^{\tilde{r}}_s).\]
Using  \eqref{q}, integration by parts, Assumption {\bf (A3(m))}, the Cauchy-Schwarz and Young inequalities, we deduce that
\begin{align*}
J^1_{\ee,t} & + J^2_{\ee,t} \leq
CA \int_0^t \Big[  \|Z_\ee(s)\|_{m+3}\big\{  [\zeta_\ee(s)]_m +  \|\zeta_\ee(s) \|_m \big\}
+ \sum_r \|\zeta_\ee(s)\|_m \; \|F_r\|_{m+2} \Big]\,  dV_s \\
&\qquad + C A \int_0^t \sum_l \|\zeta_\ee(s)\|_m \,  \|G_l(s)\|_{m+3}\,  d\langle M^l \rangle_s\\
&\leq \int_0^t \big( [\zeta_\ee (s)]_m^2 + \|\zeta_\ee(s)\|_m^2\big)\,  dV_s \\
&\qquad  + C A^2 \int_0^t \Big[ \Big(  \|Z_\ee(s)\|_{m+3}^2
+ \sum_r \|F_r\|_{m+2}^2 \Big)  dV_s + \sum_{l \geq 1} \|G_l(s)\|_{m+3}^2\,  d\langle M^l \rangle_s  \Big]\\
& \leq  \int_0^t \big( [\zeta_\ee (s)]_m^2 + \|\zeta_\ee(s)\|_m^2\big)\,  dV_s
 + C A^2 \Big( \int_0^t \|Z_\ee(s)\|_{m+3}^2\,  dV_s +
K_{m+2}(t)\Big),
\end{align*}
where the last inequality is deduced  from Assumption {\bf (A4(m+2,2))}.

\noindent The Cauchy Schwarz inequality, integration by parts and Assumption 
{\bf (A3(m+1))} imply that for fixed $r=0, \cdots, d_1$
and $l\geq 1$,
\begin{align*}
 |\big( S_l(s) \zeta_\ee(s) , L_r Z_\ee(s) & + F_r\big)_m|  + | \big( \zeta_\ee(s), L_r[S_l(s) Z_\ee(s) + G_l(s)]\big)_m| \\
& \leq C  \|\zeta_\ee(s)\|_m \big[ \|Z_\ee(s)\|_{m+3} + \|F_r\|_{m+1} + \|G_l(s)\|_{m+2} \big] .
\end{align*}
Therefore, the Burkholder Davies Gundy inequality  and Assumption {\bf (A1)} imply  that
for any stopping time $\tau\leq T$, we have:
\begin{align*}
& \EE\Big(\sup_{t\in[0, \tau]}  |J^3_{\ee, t}|^{p/2}\Big) \\
&  \leq C A^{p/2} \EE\Big( \int_0^\tau \|\zeta_\ee(s)\|_m^2 \Big[
\|Z_\ee(s)\|_{m+3}^2  + \sum_{0 \leq r\leq   d_1} \|F_r\|_{m+1}^2 \\
&\qquad
+ \sup_{s\in [0,T]} \sum_{l\geq 1}  \|G_l(s)\|_{m+2}^2\Big]  d\langle M^l\rangle_s  \Big)^{p/4}
\\
&\leq    C A^{p/2} \EE\Big[\Big(  \sup_{s\in[0, \tau]} \|Z_\ee(s)\|_{m+3}^{p/2} + \Big| \sum_{r=0}^{d_1} \|F_r\|_{m+1}^2\Big|^{p/4}\!\!
+ \sup_{s\in [0,T]} \Big| \sum_{l \geq 1} \|G_l(s)\|_{m+2}^2\Big|^{p/4} \Big)\\
&\qquad\qquad\qquad  \times
 \Big(\int_0^\tau \|\zeta_\ee(s)\|_m^2 dV_s\Big)^{p/4}\Big]\\
&\leq C A^p \; \EE \Big[ \Big(  \sup_{s\in[0,T]} \|Z_\ee(s)\|_{m+3}^{p}\Big) +  \Big|\tilde{K}  \sum_{r=0}^{d_1}
 \|F_r\|_{m+1}^2\Big|^{p/2}\!\!
+ \sup_{s\in [0,T]} \Big| \sum_{l\geq 1} \|G_l(s)\|_{m+2}^2\Big|^{p/2} \Big] \\
&\qquad +  \frac{1}{8p} \EE\Big( \sup_{s\in [0,\tau]} \|\zeta_\ee(s)\|_m^p\Big) + C \EE\int_0^\tau \|\zeta(s)\|_m^p dV_s.
\end{align*}
Using  the condition  \eqref{A6} and Theorem \ref{Hmestimates} with $m+3$, we deduce yhat
\[ \EE\Big(  \sup_{s\in[0, \tau]} |J_{\ee,t}^3|^{p/2} \Big) \leq \frac1{8p} \EE\Big(\sup_{t\in[0, \tau]} \|\zeta_\ee(t)\|_m^p\Big)
+ C \EE\int_0^\tau \|\zeta_\ee(s)\|_m^p dV_s + C A^p.\]
Therefore,
\begin{align*}
  \EE\Big\{ \sup_{t\in[0, \tau]}  \Big(J_{\ee,t}& -\int_0^t [\zeta_\ee(s)]^p_m(s) dV_s\Big)_+^{p/2} \Big\}
\leq  1/(8p) \EE\Big( \sup_{t\in[0, \tau]} \|\zeta_\ee(t)\|_m^2 \Big) \\
& + C \EE\int_0^\tau \|\zeta_\ee(s)\|_m^p dV_s
 + C A^p
+ C \EE\Big(\sup_{t\in[0, \tau]} |J^4_{\ee,t}|^{p/2}\Big).
\end{align*}
Integrating by parts we obtain
\begin{equation} \label{IPPJ4}
2J^4_{\ee,t}=\sum_{r,\tilde{r}} \big( L_{\tilde{r}} Z_\ee(t) + F_{\tilde{r}}\, ,
\, L_r Z_\ee(t) + F_r \big)_m A^r_t A^{\tilde{r}}_t
- \sum_{j=1}^{3} J^{4,j}_{\ee,t},
\end{equation}
where
\begin{align*}
J^{4,1}_{\ee,t}&=2\sum_{r,\tilde{r}} \sum_{\bar{r}} \int_0^t A^r_s A^{\tilde{r}}_s\,
 \big\<  L_{\tilde{r}}[ L_{\bar{r}} Z_\ee(s) + F_{\bar{r}}] \, , \,
L_r Z_\ee(s) + F_r \big\>_m dV^{\bar{r}}_{\ee,s}, \\
J^{4,2}_{\ee, t}&=2\sum_{r,\tilde{r}} \sum_{l\geq 1} \int_0^t  A^r_s A^{\tilde{r}}_s\,
\big(L_{\tilde{r}} [S_l(s) Z_\ee(s) + G_l(s)]\, , \, L_r Z_\ee(s) + F_r\big)_m dM^l_s,\\
J^{4,3}_{\ee, t}&= \sum_{r,\tilde{r}} \sum_{l\geq 1} \int_0^t  A^r_s A^{\tilde{r}}_s\,
\big(L_{\tilde{r}} [S_l(s) Z_\ee(s) + G_l(s)]\, , \, L_{r} [S_l(s) Z_\ee(s) + G_l(s)]\big)_m d\<M^l\>_s.
\end{align*}
Integration by part, Assumption {\bf (A3(m+2))}, the Cauchy-Schwarz and Young inequalities yield
\begin{align*}
| \big\<  L_{\tilde{r}}[ L_{\bar{r}} Z_\ee(s) + F_{\bar{r}}] \, , \,
L_r Z_\ee(s) + F_r \big\>_m |&\leq C\big[ \|Z_\ee(s)\|_{m+3}^2 + \|F_{\bar{r}}\|_{m+2}^2 + \|F_r\|_m^2\big] ,\\
| \big(L_{\tilde{r}} [S_l(s) Z_\ee(s) + G_l(s)] ,  L_r Z_\ee(s) + F_r\big)_m |&\leq C
\big[ \|Z_\ee(s)\|_{m+3}^2 + \|G_l(s)\|_{m+2}^2 + \|F_r\|_m^2\big] .
\end{align*}
Hence, using Theorem \ref{Hmestimates}, \eqref{A6}, Assumptions {\bf (A1)} and  {\bf (A4(m+2))} we deduce
\[ \EE\big( \sup_{t\in[0, \tau]} \big| J^{4,1}_{\ee,t}+J^{4,3}_{\ee,t} \big|^{p/2} \Big) \leq C A^p.\]
Finally, the Burkholder-Davies-Gundy inequality implies that
\begin{align*}
 \EE\Big( &\sup_{t\in[0, \tau]} |J^{4,2}_{\ee, t}|^{p/2}\Big) \\
& \leq C A^p \EE\Big| \int_0^\tau \big|
\big(L_{\tilde{r}} [S_l(s) Z_\ee(s) + G_l(s)]\, , \, L_r Z_\ee(s) + F_r\big)_m |^2 d\<M^l\>_s \Big|^{p/4} \leq C A^p.
\end{align*}
Hence, $\EE\big( \sup_{t\in[0, \tau]} |J^4_{\ee,t}|^{p/2}\big) \leq C A^p$, which concludes the proof.
\end{proof}

Using Lemmas \ref{Cm}-\ref{J-int}, we now prove Theorem \ref{Thsplitting1} for time-independent coefficients. 

\noindent{\it Proof of Theorem \ref{Thsplitting1}}
Apply the operator $D^\alpha$ to both sides of \eqref{zetaeebis} and use the It\^o formula for
 $\|D^\alpha \zeta_\ee(t)\|_0^2$. This yields
\begin{align*}
d\|\zeta_\ee(t)\|_m^2=& 2\sum_{|\alpha|\leq m} \sum_r
 \big[ \big\< \zeta_\ee(t),L_r \zeta_\ee(t)\big\>_m \rho_{\ee,t}^r dV_t +
 \big(\zeta_\ee(t), L_r Z_\ee(t)+F_r\big)_m dA^r_t \big] \\
&  + \sum_{|\alpha|\leq m} \sum_{l\geq 1}
\big[ \|S_l(t) \zeta_\ee(t)\|_m^2 q^l_t dV_t +2\big( \zeta_\ee(t), S_l(t)\zeta_\ee(t)\big)_m dM^l_t\big],
\end{align*}
where $\< Z,\zeta\>_m$ denotes the duality between $H^{m-1}$ and $H^{m+1}$ which extends the
scalar product in $H^m$. Using  \eqref{p} we deduce that
\[ d\|\zeta_\ee(t)\|_m^2\leq -2 [\zeta_\ee(t)]_m^2 dV_t + C \|\zeta_\ee(t)\|^2_m dV_t + 2 dJ_{\ee,t}
+ 2\!\sum_{l\geq 1}\! \big(\zeta_\ee(t), S_l(t) \zeta_\ee(t)\big)_m dM^l_t,\]
where $J_{\ee,t}$ is defined by \eqref{Jeet}.
Using \eqref{majoSl} we deduce that
$\big|\big(\zeta_\ee(t), S_l(t) \zeta_\ee(t)\big)_m \big| \leq C \|\zeta_\ee(t)\|_m^2$.  Thus Lemma \ref{J-int},
the Burkholder Davies Gundy inequality and Assumption {\bf (A1)} yield for any stopping time $\tau \leq T$
\begin{align*}
&\EE\Big( \sup_{t\in[0, \tau]} \|\zeta_\ee(t)\|^p_m\Big) \leq C \EE \|Z_{1,0} - Z_{0,0}\|_m^p
 + p \EE \Big( \sup_{t\in[0, \tau]} J_{\ee,t} - \int_0^t [\zeta_\ee(s)]_m^2 dV_s \Big)_+^{p/2} \\
&\qquad  + C_p \EE\Big| \int_0^\tau \|\zeta_\ee(s)\|_m^4 d V_s\Big|^{p/4}
+  C_p \EE\Big| \int_0^\tau \|\zeta_\ee(s)\|_m^2 d V_s\Big|^{p/2}
\\
 & \quad \leq C \EE \|Z_{1,0} - Z_{0,0}\|_m^p + \frac{1}{4} \EE\Big( \sup_{t\in[0, \tau]} \|\zeta_\ee(t)\|_m^p
 \Big) + C\Big( A^p + \EE\int_0^\tau \|\zeta_\ee(s)\|_m^p dV_s \Big) \\
 &\qquad  + C \EE\Big[ \sup_{t\in[0, \tau]} \|\zeta_\ee(t)\|_m^{p/2} \Big( \int_0^\tau \|\zeta_\ee(s)\|_m^2 dV_s \Big)^{p/4} \Big]
+ C_p \EE \int_0^\tau \|\zeta_\ee(s)\|_m^p d V_s
\\
 & \quad \leq C \EE \|Z_{1,0} - Z_{0,0}\|_m^p + C A^p +  \frac{1}{2} \EE\Big( \sup_{t\in[0, \tau]} \|\zeta_\ee(t)\|_m^p \Big)
+ C \EE\Big( \int_0^\tau \|\zeta_\ee(s)\|_m^p dV_s \Big),
\end{align*}
where the last upper estimate follows from the Young inequality. \\
Let $\tau_N = \inf\{ t : \|\zeta_\ee(t)\|_m^p \geq N\} \wedge T$; then the Gronwall Lemma implies that
\[  \EE\Big( \sup_{t\in[0, \tau_N]} \|\zeta_\ee(t)\|_m^p \Big) \leq C  \big( \EE\|Z_{1,0} - Z_{0,0}\|_m^p
+ A^p\big) .
\]
Letting $N\to \infty$ concludes the proof.
\qed

\subsection{Case of the time dependent coefficients}\label{timedepen}
In this section, we prove a convergence result similar to that in Theorem \ref{Thsplitting1}
when the coefficients of the operators depend on time. Integration by parts in Lemma \ref{J-int}
will give extra terms, which require more assumptions to be dealt with.

\noindent {\bf Assumption (A5(m))} There exists an integer number  $d_{2}$,  an $\big({\mathcal F}_{t}\big)$-continuous martingale
$N_{t}= (N_{t}^{1}, \cdots, N_{t}^{d_{2}})$ and,  for each $\gamma=0, \cdots, d_{2}$
a bounded predicable process   $  h_{\gamma} : \Omega\times (0,T]\times \mathbb{R}^d\to \mathbb{R}^N$
for some $N$ depending on $d$ and $d_{1}$
 such that
\begin{eqnarray*}
 h_{\gamma}(t,x) &:= & ( a^{j,k}_{\gamma,r}(t,x), b^{j,k}_{\gamma,r}(t),
a^{j}_{\gamma,r}(t,x), a_{\gamma,r}(t,x), b_{\gamma,r}(t,x),  F_{\gamma,r}(t,x) ;\\
 && 1\leq j,k \leq d, \;0\leq r  \leq d_{1},\;
1\leq  \gamma \leq  d_{2}),
\end{eqnarray*}
for  some symmetric  non negative matrices $(a^{j,k}_{\gamma,r}(t,x), j,k=1, \cdots, d)$
and $(b^{j,k}_{ \gamma,r}(t),$ $ j,k=1,\cdots ,d)$.
Furthermore, we suppose that
for every $\omega\in \Omega$ and $t\in [0,T]$, the maps  $h_{\gamma}(t,\cdot)$ are of class
${\mathcal C}^{m+1}$ such that for some constant $K$ we have
$|D^{\alpha}h_{\gamma}(t,\cdot)|\leq K$  for any multi-index $\alpha$ with $|\alpha| \leq m+1$ and such that
for $t\in [0,T]$,
\begin{eqnarray*}
\sum_{\gamma=1}^{d_{2}}  d\langle N^{\gamma}\rangle_{t} &\leq& dV_{t}, \\
h(t,x)&=&h(0,x) + \int_{0}^{t} h_{0}(s,x) dV_{s}
+ \sum_{\gamma=1}^{d_{2}}\int_{0}^{t} h_{\gamma}(s,x) dN^{\gamma}_{s}.
\end{eqnarray*}
For $\gamma=0, \cdots, d_{2} $,
$r=0, \cdots, d_{1}$,  let $L_{\gamma,r}$ be the  time dependent  differential operator defined by:
\begin{eqnarray*}
L_{\gamma,r} Z(t,x)&=&\sum_{j,k=1}^{d} D_k\Big( \big[a^{j,k}_{\gamma,r}(t,x) + i b^{j,k}_{\gamma,r}(t)\big]
 D_jZ(t,x) \Big)
+ \sum_{j=1}^{d} a^j_{\gamma,r}(t,x) D_j Z(t,x)   \\
&& + \big[ a_{\gamma,r}(t,x)+ib_{\gamma,r}(t,x)\big] Z(t,x) .
\end{eqnarray*}
For $r=0, \cdots, d_1$, let
\[ L_rZ(t,x)=L_r(0) Z(0,x) + \int_0^t L_{0,r} Z(s,x) dV_s + \sum_{\gamma = 1}^{d_2} L_{\gamma,r} Z(s,x) dN^\gamma_s,\]
and $F_r(t,x)=F_r(0,x)+\int_0^t F_{0,r}(s,x) dV_s + \sum_{\gamma = 1}^{d_2} F_{\gamma,r}(s,x) dN^\gamma_s$.
We then have the following abstract convergence result which extends Theorem\ref{Thsplitting1}.
\begin{theorem}\label{Thsplitting2}
Let  Assumptions {\bf (A(1)), (A(2))}, {\bf (A3(m+3))}, {\bf (A4(m+3,p))} and {\bf (A5(m))} be satisfied
and suppose that 
\begin{equation} \label{A6bis}
\EE\Big( \sup_{t\in [0,T]} \Big| \sum_{r=0}^{d_1} \|F_r(t)\|_{m+1}^2\Big|^{p/2}
+ \sup_{t\in [0,T]} \Big| \sum_{l\geq 1} \|G_{m+2}(t)\|^2_{m+2} \Big|^{p/2} <\infty.
\end{equation}
Then there exists some constant $C>0$ such that
\begin{equation}\label{ineq-rate} \EE \Big( \sup_{t\in [0,T]}   \| Z_{1}(t) - Z_{0}(t)\|_{m}^{p}) \leq
 C \Big(\EE\big( \|Z_{1}(0) - Z_{0}(0)\|_{m}^{p} \big) + A^{p}\Big)  .
 \end{equation}
\end{theorem}
\begin{proof} 
Since Lemmas \ref{Cm} and \ref{tech} did not depend on the fact that the coefficients
are time-independent, only Lemma \ref{J-int} has to be extended.
For $t\in [0,T]$, let
\[ J_{\ee,t}=\sum_{r=0}^{d_1} \int_0^t \big( \zeta_\ee(s) \, ,\, L_r(s) Z_{\ee}(s) + F_r(s) \big)_m
dA^r_s.\]
Since $A^{r}_{0}=0$ for $r=0, \cdots, d_{1}$, the integration by parts formula \eqref{IPPJee} has
to be replaced by
\[
 J_{\ee,t} = \sum_{r=0}^{d_1}  \big( \zeta_\ee(t) \, ,\, L_r(t) Z_\ee(t) + F_r(t) \big)_m A^r_t - \sum_{k=1}^{7} J^k_{\ee,t}, \;t\in [0,T] ,
\]
where the additional terms on the right hand-side are defined  for $t\in [0,T] $ as follows:
\begin{align*}
J^5_{\ee,t} = & \sum_{r} A^r_s \big( \zeta_\ee(s) , L_{0,r} Z_\ee(s) + F_{0,r} \big)_m
 dV_s, \\
J^6_{\ee,t} = & \int_0^t \sum_r A^r_s \sum_{\gamma=1}^{d_{2}}
 \big( \zeta_\ee(s)   , L_{\gamma,r} Z_{\ee}(s) + F_{\gamma,r}(s)  \big)_m dN^{\gamma}_{s}, \\
J^7_{\ee,t}=&  \int_0^t \sum_r  A^r_s  \sum_{l\geq 1} \sum_{\gamma=1}^{d_{2}}
\big( S_l(s) \zeta_\ee(s)  , L_{\gamma,r} Z^{\ee}(s) + F_{\gamma,r}(s)  \big)_m
d\langle M^{l}, N^{\gamma}\rangle_{s}.
\end{align*}
 Arguments similar to those used in the proof of Lemma \ref{J-int}, using integration by parts and
 the regularity assumptions of the coefficients,  prove that for $k=5,6$ there exists a constant $C>0$
 such that for any stopping time $tau\leq T$ we have:
\begin{align*} \EE\Big( \sup_{t\in [0,\tau]}  | J^{k}_{\ee,t}  |^{p/2 }\Big) & \leq
C A^{p/2} \EE\Big(\sup_{t\in [0,\tau]}   \|\zeta_{\ee}(t)\|_{m}^{p/2}
\sup_{t\in [0,\tau]}    \big(      \|Z_{\ee}(t) \|_{m_+2} +C \big)^{p/2 }\Big) \\
& \leq \frac{1}{24p } \EE\Big(  \sup_{t \in [0,\tau]}  \|\zeta_{\ee}(t)\|_{m}^{p} \Big) + C A^{p},
\end{align*}
 where the last inequality  follows from the Young inequality. Furthermore, the Burkholder-Davies-Gundy
 inequality and the upper estimates of the quadratic variations of the martingales $N^{\gamma} $ yield
 for every $\gamma = 1, \cdots, d_{2}$, $r=0,  \cdots, d_{1}$ and $\tau\in [0,T]$,
 \begin{align*} \EE\Big(\sup_{t\in [0,\tau]} | J^{7}_{\ee,t}  |^{p/2 } \Big)  \leq &  C
 \EE\Big( \int_{0}^{\tau}  \big| A^{r}_{s} \big(  \zeta_{\ee}(s) , L_{\gamma,r} Z_{\ee}(s) +
 F_{\gamma,r}(s) \big)_{m}\big|^{2} dV_{s} \Big)^{p/4} \\
 &  \leq C  A^{p/2} \EE\Big(\sup_{t\in [0,\tau]}   \|\zeta_{\ee}(t)\|_{m}^{p/2}
\sup_{t\in [0,\tau]}    \big(      \|Z_{\ee}(t) \|_{m_+2} +C \big)^{p/2 }\Big).
 \end{align*}
 Hence, the proof will completed by  extending the upper estimate \eqref{IPPJ4}
 as follows:
 \[
2J^4_{\ee,t}=\sum_{r,\tilde{r}} \big( L_{\tilde{r}} Z_\ee(t) + F_{\tilde{r}}(t)\, ,
\, L_r Z_\ee(t) + F_r (t)\big)_m A^r_t A^{\tilde{r}}_t
- \sum_{j=1}^{7} J^{4,j}_{\ee,t},
\]
where for $j=4, \cdots,7$ we have:
\begin{align*}
J^{4,4}_{\ee,t}&=2\sum_{r, \tilde{r}}  \int_0^t A^r_s A^{\tilde{r}}_s\,
 \big(  L_{\tilde{r},0 }(s) Z_\ee(s) + F_{\tilde{r},0}(s) \, ,\, L_r(s) Z_\ee(s) + F_r(s) \big)_m  dV_s , \\
J^{4,5}_{\ee, t}&=\sum_{r,\tilde{r}} \sum_{\gamma, \tilde{\gamma}}  \int_0^t  A^r_s A^{\tilde{r}}_s\,
\big(L_{\tilde{\gamma}, \tilde{r}}(s)  Z_\ee(s) + F_{\tilde{\gamma},\tilde{r}}(s) \, , \,
 L_{\gamma ,r}  Z_\ee(s) + F_{\gamma, r} (s)\big)_m d\langle N^{\tilde{\gamma}} , N^\gamma\rangle_s ,\\
J^{4,6}_{\ee, t}&= 2 \sum_{r,\tilde{r}} \sum_\gamma \sum_{l\geq 1} \int_0^t  A^r_s A^{\tilde{r}}_s\,
\big(L_{\gamma, \tilde{r}} (s) Z_\ee(s) + F_{\gamma, \tilde{r}}(s) \, ,\\
&\qquad\qquad  \, L_{r}(s) [ S_l(s) Z_\ee(s)
+ G_l(s) ] \big)_m d\<N^\gamma, M^l\>_s,\\
J^{4,7}_{\ee, t}&= 2 \sum_{r,\tilde{r}} \sum_\gamma  \int_0^t  A^r_s A^{\tilde{r}}_s\,
\big(L_{\gamma, \tilde{r}}(s)  Z_\ee(s) + F_{\gamma, \tilde{r}}(s) \, , \, L_{r} (s) Z_\ee(s) + F_r(s) \big)_m dN^\gamma_s .
\end{align*}
We obtain upper estimates of the terms $\EE\big( \sup_{t\in [0,\tau]} | J^{4,k}_{\ee,t}|^{p/2}\big)$ for $k=4, \cdots, 7$
 by arguments similar to that used for $k=1,\cdots, 3$, which implies $\EE\big( \sup_{t\in [0,\tau]} |J^4_{\ee,t}|^{p/2}
 \big) \leq C A^p$. This concludes the proof.
\end{proof}

\section{Speed of convergence for the spliting method}    \label{sspeed}
The aim of this section is to show how the abstract convergence results obtained in Section \ref{speedabstract}
yield  the convergence of a splitting method and extends the corresponding results from \cite{GK}. The proof, which is
very similar to that in \cite{GK} is briefly sketched for the reader's convenience.

\noindent {\bf Assumption (A)} {\it
For $r=0, \cdots, d_1$, let  $L_r$  be defined by \eqref{Lr} and   for $l\geq 1$
let $S_l$  be defined  by \eqref{Sl}. Suppose that the Assumptions {\bf (A2)} and
{\bf(A3(m+3))} are satisfied, and that  for every $\omega\in \Omega$  and $r,l$, $F_r(t)=F_r(t,\cdot)$
is a weakly continuous $H^{m+3}$-valued function and  $G_l(t)=G_l(t,\cdot)$ is a
weakly continuous $H^{m+4}$-valued function. Suppose furthermore that the initial condition $Z_0\in L^2(\Omega ; H^{m+3})$
is ${\mathcal F}_0$-measurable,  that $F_r$ and $G_l$ are predictable and that for
some constant $K$ one has
\[
\EE\Big( \sup_{t\in [0,T]} \sum_{r=0}^{d_1}  \|F_r(t,\cdot)\|_{m+3}^p
+ \sup_{t\in [0,T]} \sum_{l\geq 1}  \|G(t,\cdot)\|_{m+4}^p + \|Z_0\|_{m+3}^p\Big) \leq K.
\]
Let $V^0=(V^0_t, t\in [0,T])$ be a predictable, continuous increasing process such that $V^0_0=0$ and that there exists
a constant $K$ such that $V^0_T + \sum_{l\geq 1} \langle M^l\rangle_T \leq K$. Finally suppose that
the following stochastic parabolicity condition holds.\\ For every $(t,x)\in [0,T]\times \RR^d$,
every $\omega \in \Omega$ and every $\lambda \in \RR^d$,
\[ \sum_{j,k=1}^{d} \lambda_j \lambda_k \Big[ 2 a_0^{j,k}(t,x) dV^0_t + \sum_{l\geq 1}
 \s_l^j(t,x) \s_l^k(t,x)  d\langle M\rangle_t \Big] \geq 0
\]
in the sense of measures on $[0,T]$.}

Let  $Z$ be the process solution to the evolution equation:
\begin{align}\label{eqZ}
dZ(t,x) = &  \big( L_0 Z(t,x) + F_0(t,x)\big) dV^0_t + \sum_{r=1}^{d_1} \big( L_r Z(t,x) + F_r(t,x)\big) dt
\nonumber  \\
&  + \sum_{l\geq 1} \big( S_l Z(t,x) + G_l(t,x)\big) dM^l_t
\end{align}
with the initial condition $Z(0,\cdot)=Z_0$. Then Theorem \ref{Hmestimates} proves the existence and uniquness
 of the solution to \eqref{eqZ}, and that
\[ \EE\Big( \sup_{t\in [0,T]} \|Z(t)\|_{m+3}^p \Big) \leq C\]
for some constant $C$ which depends only on $d$, $d_1$,  $K$, $m$, $p$ and $T$.

For every integer $n\geq 1$ let ${\mathcal T}_n  = \{ t_i:=iT/n , i=0, 1, \cdots, n\}$
denote a grid on the interval $[0,T ]$ with constant mesh
$\delta = T/n$. For $n\geq 1$, let $Z^{(n)}$ denote the approximation of $Z$ defined for $t\in {\mathcal T}_n$
using the following splitting method: $Z^{(n)}n(0)=0$ and for $i=0, 1, \cdots, n-1$, let
\begin{equation}  \label{discret}
Z^{(n)} (t_{i+1}) := P_\delta^{(d_1)} \cdots P_\delta^{(2)} P_\delta^{(1)} Q_{t_i, t_{i+1}} Z^{(n)}(t_i),
\end{equation}
where for $r=1, \cdots, d_1$ and $t\in [0,T]$, $P^{(r)}_t \psi$ denotes  the solution $\zeta_r$  of the evolution equation
\[ d\zeta_r(t,x) = \big( L_r \zeta_r(t,x) + F_r(t,x) \big) dt \quad \mbox{\rm and } \zeta_r(0,x)=\psi(x),\]
and for $s\in [0,t]\leq T$, $Q_{s,t}\psi$ denotes the solution $\eta$ of the evolution equation defined on $[s,T]$ by
the "initial" condition $ \eta(s,x)=\psi(x)$ and for $t\in [s,T]$ by:
\[ d\eta(t,x) = \big( L_0 \eta(t,x) + F_0(t,x) \big) dV^0_t + \sum_{l\geq 1} \big( S_l \eta(t,x) + G_l(t,x)\big) dM^l_t .\]
The following theorem gives the speed of convergence of this approximation.
\begin{theorem}\label{thspeed1}
Let $a^{j,k}_r, b^{j,k}_r, a^j_r, a_r, b_r, \s_l^j,  \s_l, \tau_l, F_r, G_l$ satisfy the Assumption {\bf (A)}.
Suppose that $a^{j,k}_r, b^{j,k}_r, a^j_r, a_r, b_r$ and $F_r$ are time-independent.
Then there exists a constant $C>0$ such that
\[  \EE\Big( \sum_{t\in {\mathcal T}_n} \|Z^{(n)}(t) -Z(t)\|_m^p \Big) \leq C n^{-p},\;\; \mbox{for every $n\geq 1$}. \]
\end{theorem}
\begin{proof}
Let $d'=d_1+1$ and let us introduce the following time change:
\[ \kappa(t) = \left\{
\begin{array}{ll}
 0& \mbox{\rm for } t\leq 0,\\
t-k\delta d_1& \mbox{\rm for } t\in [k d' \delta , (kd'+1)\delta), \; k=0, 1, \cdots, n-1 , \\
(k+1)\delta & \mbox{\rm for } t\in [(k d'+1)  \delta , (k+1) d'\delta), \; k=0, 1, \cdots , n-1 .
\end{array}
\right.
\]
Let, for evry $t\in [0,T]$,
\[\begin{array}{lll}
\tilde{M}^l(t) = M^l_{\kappa(t)}, & \tilde{\mathcal F}_t={\mathcal F}_{\kappa(t)}, &\tilde{V}_{t,0}^0 = \tilde{V}_{t,1}^0
= V^0_{\kappa(t)}, \\
 \tilde{V}^r_{t,0}= \kappa(t), & \tilde{V}^r_{t,1} = \kappa(t-r\delta)& \mbox{\rm for } r=1, 2, \cdots,  d_1 .
\end{array} \]
For $\ee=0,1$, consider the evolution equations with the same initial condition $Z_0(0,x)=Z_1(0,x)=Z_0(x)$ and
\begin{equation} \label{Zeesplit}
dZ_\ee(t) = \sum_{r=0}^{d_1} \big( L_r Z_\ee(t) + F_r\big) d\tilde{V}_{t,\ee}^r +
\big( S_l Z_\ee(t) + G_l\big) d\tilde{M}^l_t.
\end{equation}
One easily checks that the Assumptions {\bf (A1), (A2)}, {\bf (A3(m+3)), (A4(m+3,p))} are satisfied with
the martingales $\tilde{M}_l$ and the increasing processes $\tilde{V}^r_{\ee,t}$ for $\ee=0,1$ and $r=0, 1, \cdots, d_1$.
Therefore, Theorem \ref{Hmestimates} implies that for $\ee=0,1$, the equation \eqref{Zeesplit} has a unique solution.
Furthermore, since condition \eqref{A6} holds, Theorem \ref{Thsplitting1} proves the existence of a constant $C$ such that
\[\EE\Big( \sup_{t\in [0, d'T]} \|Z_1(t) - Z_0(t)\|_m^p\Big) \leq C \sup_{t\in [0, d'T]} \max_{ 1\leq r\leq d_1} |\kappa(t+r\delta)
- \kappa(t)|^p = C T^p n^{-p}.
\]
Since by construction, we have $Z_0(d't) = Z(t)$ and $Z_1(d't)=Z^{(n)}(t)$ for $t\in {\mathcal T}_n$,
this concludes the proof.
\end{proof}
Note that the above theorem yields a splitting method for the following linear Schr\"odinger equation on $\RR^d$:
\begin{align*}
dZ(t,x) = &  \Big(  i\Delta Z(t,x) + \sum_{j=1}^{d} a^j(x) D_j Z(t,x) + F(x)\Big) dt \\
&  + \sum_{l\geq 1}
\big[ \big( \sigma_l(x) + i\tau_l(x) \big) Z(t,x) + G_l(x)\big] dM^l_t,
\end{align*}
where $a^j$, $F$ (resp.  $ \sigma_l$,  $\tau_l$ and $G_l$) belong to $H^{m+3}$ (resp.  $H^{m+4}$). Indeed, this
model is obtained with $a^{j,k}=\sigma_l^j=0$  and $b^{j,k}=1$  for $j,k=1, \cdots, d$ and $l\geq 1$.

Finally, Theorem \ref{Thsplitting2} yields the following theorem for the splitting method in the case of
time demendent coefficients. The proof, similar to that of Theorem \ref{thspeed1},  will be omitted; see also
\cite{GK}, Theorem 5.2.
\begin{theorem} \label{thspeed2}
Let $a^{j,k}_r, b^{j,k}_r, a^j_r,  \s_l^j,  \s_l, \tau_l, F, G_l$ satisfy  Assumptions  {\bf (A)}
 and {\bf (A5(m))}. For every integer $n\geq 1$ let $Z^{(n)}$ be defined by \eqref{discret}
when the operators $L_r$, $S_l$, the processes $F_r$ and $G_l$ depend on time in a predictable way.
 Then there exists a constant $C>0$ such that for every $n\geq 1$, we have:
\[  \EE\Big( \sum_{t\in {\mathcal T}_n} \|Z^{(n)}(t) -Z(t)\|_m^p \Big) \leq C n^{-p}. \]
\end{theorem}

\end{document}